\documentclass[10pt]{article}
\usepackage{amsmath,amssymb,amsthm}
\usepackage{graphicx,color}
\usepackage[margin=1in]{geometry}
\usepackage[affil-it]{authblk}

\newtheorem{proposition}{Proposition}

\theoremstyle{definition}
\newtheorem{definition}{Definition}

\newmuskip\pFqskip
\mathchardef\pFcomma=\mathcode`, % keep a copy of the comma

\newcommand*\pFq[5]{%
  \begingroup
  \begingroup\lccode`~=`,
    \lowercase{\endgroup\def~}{\pFcomma\mkern\pFqskip}%
  \mathcode`,=\string"8000
  {}_{#1}F_{#2}\biggl[\genfrac..{0pt}{}{#3}{#4};#5\biggr]%
  \endgroup
}

\newcommand*\pGq[5]{%
  \begingroup
  \begingroup\lccode`~=`,
    \lowercase{\endgroup\def~}{\pFcomma\mkern\pFqskip}%
  \mathcode`,=\string"8000
  {}_{#1}\phi_{#2}\biggl[\genfrac..{0pt}{}{#3}{#4}\bigg|#5\biggr]%
  \endgroup
}
\numberwithin{equation}{section}
\allowdisplaybreaks

\title{Bivariate Bannai-Ito polynomials}
\author[1]{Jean-Michel Lemay}
\affil[1]{Centre de Recherches Math\'ematiques, Universit\'e de Montr\'eal, C.P. 6128,\protect\\ Succ. Centre-ville, Montr\'eal, QC, Canada, H3C 3J7 \authorcr E-mail address: jean-michel.lemay.1@umontreal.ca, vinet@crm.umontreal.ca}
\author[1]{Luc Vinet}
\date{}

\begin{document}
\maketitle

%----------------------------------------------------------------------------------------
\begin{abstract}

\noindent A two-variable extension of the Bannai-Ito polynomials is presented. They are obtained via $q\to-1$ limits of the bivariate $q$-Racah and Askey-Wilson orthogonal polynomials introduced by Gasper and Rahman. Their orthogonality relation is obtained. These new polynomials are also shown to be multispectral. Two Dunkl shift operators are seen to be diagonalized by the bivariate Bannai-Ito polynomials and 3- and 9-term recurrence relations are provided.

\end{abstract}

\section*{Introduction}

In their classification of P- and Q-polynomial association schemes \cite{BannaiIto}, Bannai and Ito identified a new 4-parameter family of orthogonal polynomials that now bear their names. They provided the explicit expressions of these polynomials and observed that they correspond to a $q\to-1$ limit of the $q$-Racah polynomials. The understanding of the Bannai-Ito polynomials has considerably increased in recent years. Of particular relevance to the present study is the fact that they have been shown \cite{DunklShift} to also arise as $q\to-1$ limits of the Askey-Wilson polynomials. The Bannai-Ito polynomials are now known to be multispectral : they are eigenfunctions of the most general first order shift operator of Dunkl type that preserves the space of polynomials of a given degree \cite{DunklShift}. They have been identified with the non-symmetric Wilson polynomials \cite{nonsym} and are essentially the Racah coefficients of the Lie superalgebra $\mathfrak{osp}(1|2)$ \cite{RacahCoef}. They have moreover found various applications beyond algebraic combinatorics especially in the context of superintegrable and exactly solvable models \cite{laplacedunkl,diracdunkl,BIalgebra_2sphere}. The Bannai-Ito polynomials and their kernel partners, the complementary Bannai-Ito polynomials admit various multispectral families of orthogonal polynomials as descendants and special cases and thus sit at the top of a $q=-1$ analog of the Askey-scheme \cite{family,bigjac,bochner,dual,DunklShift,CBI,ChiharaP}.

The extension to many variables of the theory of univariate orthogonal polynomials is obviously of great interest. There are two major directions in this broad topic (see \cite{DunklXu, Marcellan} for instance). One involves the theory of symmetric functions \cite{Mac} and has the Macdonald and Koornwinder polynomials associated to root systems as main characters. The other works through the coupling of univariate polynomials and features the multivariable extension of the Racah and Wilson polynomials and their descendants introduced by Tratnik \cite{Tratnik1,Tratnik2}. We shall focus on the latter area in the following. A key feature of Tratnik's construction is that the multivariate orthogonality relation is obtained by induction on the univariate one. Iliev and Geronimo have shown that these Tratnik polynomials are multispectral \cite{GeroIliev,Iliev}. Their $q$-generalizations have been discovered by Gasper and Rahman who thus provided multivariable extension of the $q$-Racah and Askey-Wilson polynomials and in so doing of the entire $q$-scheme \cite{GR1,GR2}.

The goal of the present paper is to initiate a multivariable extension for the $q=-1$ scheme. Specifically, we introduce a bivariate extension of the Bannai-Ito polynomials and provide various structure relations. These polynomials are defined in formula \eqref{bivBI}.

The paper will be comprised of three main sections. We begin with a review of the Bannai-Ito polynomials and their structure relations. In section 2, we define the bivariate Bannai-Ito polynomials from a $q\to-1$ limit of Gasper and Rahman's two-variables $q$-Racah polynomials. The truncation conditions are examined and the orthogonality relation is obtained. In section 3, we obtain an untruncated definition for the bivariate Bannai-Ito polynomials from a $q\to-1$ limit of two-variable Askey-Wilson polynomials. This has the benefit of expressing the multispectrality relations in terms of operators which act directly on the variables instead of on the points of the orthogonality grid. A connection between the two definitions is established and the multispectrality relations for the polynomials are derived. Remarks and open questions are discussed in the conclusion.

%----------------------------------------------------------------------------------------
\section{Univariate Bannai-Ito polynomials}

The monic Bannai-Ito polynomials $B_n(x;\rho_1,\rho_2,r_1,r_2)$, or $B_n(x)$ for short, depend on 4 parameters $\rho_1,\rho_2,r_1,r_2$ and are symmetric with respect to the $\mathbb{Z}_2 \times \mathbb{Z}_2$ group of transformations generated by $\rho_1 \leftrightarrow \rho_2$ and $r_1 \leftrightarrow r_2$. Explicitly, this means that the BI polynomials verify
\begin{align}
 B_n(x;\rho_1,\rho_2,r_1,r_2)=B_n(x;\rho_2,\rho_1,r_1,r_2)=B_n(x;\rho_1,\rho_2,r_2,r_1)=B_n(x;\rho_2,\rho_1,r_2,r_1).    
\end{align}
We denote by $g$ the combination of parameters 
\begin{align}
 g=\rho_1+\rho_2-r_1-r_2.
\end{align}
Throughout this section, it will be convenient to write integers as follows
\begin{align} \label{paritynotation}
n=2n_e+n_p, \qquad n_p\in\{0,1\}, \quad n\in\mathbb{N}.           
\end{align}
With these notations, the Bannai-Ito polynomials can be expressed in terms of two generalized hypergeometric series 
\begin{align} \label{hypergeo-expr}
%\begin{aligned}
\frac{1}{\eta_n}B_n(x;\rho_1,\rho_2,r_1,r_2) =\ &\pFq{4}{3}{\scriptstyle -n_e, n_e+g+1, x-r_1+\frac12, -x-r_1+\frac12}{\scriptstyle 1-r_1-r_2, \rho_1-r_1+\frac12, \rho_2-r_1+\frac12}{1} \\[0.5em]
 &+ \frac{(-1)^n(n_e\!+\!n_p\!+\!g n_p)(x\!-\!r_1\!+\!\tfrac12)}{(\rho_1-r_1+\tfrac12)(\rho_2-r_1+\tfrac12)} \pFq{4}{3}{\scriptstyle -n_e-n_p+1, n_e+n_p+g+1, x-r_1+\frac32, -x-r_1+\frac12}{\scriptstyle 1-r_1-r_2, \rho_1-r_1+\frac32, \rho_2-r_1+\frac32}{1} \notag
%\end{aligned}
\end{align}
where the normalization coefficient is given by 
\begin{align} \label{eta}
 \eta_n = (-1)^{n} \frac{(\rho_1-r_1+\tfrac12)_{n_e+n_p}(\rho_2-r_1+\tfrac12)_{n_e+n_p}(1-r_1-r_2)_{n_e}}{(n_e+g+1)_{n_e+n_p}}.
\end{align}
The expression \eqref{hypergeo-expr} can be obtained from a $q\to-1$ limit of the $q$-Racah polynomials \cite{BannaiIto} and also from a $q\to-1$ limit of the Askey-Wilson polynomials \cite{DunklShift}. Note that the two hypergeometric functions appearing in \eqref{hypergeo-expr} are almost identical except for two $+1$ shifts in the upper parameter row and two in the lower row.   

The $B_n(x)$ satisfy the three-term recurrence relation
\begin{gather}
\label{Recurrence-Relation}
xB_n(x) = B_{n+1}(x)+(\rho_1-A_n-C_n)B_{n}(x)+A_{n-1}C_{n}B_{n-1}(x),
\end{gather}
with the initial conditions $B_{-1}(x)=0$ and $B_{0}(x)=1$.
The recurrence coef\/f\/icients $A_n$ and $C_n$ are given by
\begin{gather}\label{Coeff-An}
\begin{aligned}
A_n&=
\begin{cases}
\dfrac{(n+2\rho_1-2r_1+1)(n+2\rho_1-2r_2+1)}{4(n+g+1)}, & n~\text{even},
\vspace{1mm}\\
\dfrac{(n+2g+1)(n+2\rho_1+2\rho_2+1)}{4(n+g+1)}, & n~\text{odd},
\end{cases}
\\[1em] 
&C_n =
\begin{cases}
-\dfrac{n(n-2r_1-2r_2)}{4(n+g)}, & n~\text{even},
\vspace{1mm}\\
-\dfrac{(n+2\rho_2-2r_2)(n+2\rho_2-2r_1)}{4(n+g)}, & n~\text{odd}.
\end{cases}
\end{aligned}
\end{gather}
It can be seen from the above relations that the positivity conditions $u_n = A_{n-1}C_n >0$ cannot be satisfied for all $n\in\mathbb{N}$. This comes from the fact that $C_n$ becomes negative for large $n$. It follows that the Bannai-Ito polynomials can only form a finite set of orthogonal polynomials for which the conditions $u_n>0$, $n=1,2,\dots,N$ are verified. This requires that the parameters realize a truncation condition for which 
\begin{align} \label{truncationcondition}
u_0=u_{N+1}=0.
\end{align}
We call the integer $N$ the truncation parameter.

If these conditions are fulf\/illed, the BI polynomials $B_{n}(x)$ satisfy the discrete
orthogonality relation
\begin{gather}
\sum_{k=0}^{N}w_{k}B_{n}(x_k)B_{m}(x_k)=h_n\delta_{nm},
\end{gather}
with respect to a positive set of weights~$w_k$.
The orthogonality grid ~$x_k$ corresponds to the simple roots of the polynomial $B_{N+1}(x)$.
The explicit formulas for the weight function $w_{k}$ and the grid points~$x_k$ depend on the
parity of $N$ and more explicitly on the realization of the truncation condition~$u_{N+1}=0$.

If $N$ is even, it follows from~\eqref{Coeff-An} that the condition $u_{N+1}=0$ is tantamount to one of the following requirements associated to all possible values of $j$ and $\ell$ :
\begin{align} \label{Tcond1}
i)~ r_j - \rho_\ell = \frac{N+1}{2}, \qquad j,\ell \in\{1,2\}.
\end{align}
Note that the four possibilities coming from the choices of $j$ and $\ell$ are equivalent since the polynomials $B_n(x)$ are invariant under the exchanges $\rho_1\leftrightarrow\rho_2$ and $r_1\leftrightarrow r_2$.
To make the formulas explicit, fix $j=\ell=1$. Then the grid points have the expression
\begin{gather}
\label{gridi}
x_k=(-1)^{k}(k/2+\rho_1+1/4)-1/4,
\end{gather}
 for $k=0,\dots,N$ and using \eqref{paritynotation} the weights take the form
\begin{gather}
\label{Ortho-Weight}
w_{k}=\frac{(-1)^{k}}{k_e!}
\frac{(\rho_1-r_1+1/2)_{k_e+k_p}(\rho_1-r_2+1/2)_{k_e+k_p}(\rho_1+\rho_2+1)_{k_e}(2\rho_1+1)_{k_e}}
{(\rho_1+r_1+1/2)_{k_e+k_p}(\rho_1+r_2+1/2)_{k_e+k_p}(\rho_1-\rho_2+1)_{k_e}},
\end{gather}
where $(a)_{n}=a(a+1)\cdots(a+n-1)$ denotes the Pochhammer symbol. The normalization factors are 
\begin{align} \label{Ortho-Norm}
 h_{n} = \frac{n_e! N_e! (1+2\rho_1)_{N_e} (1+\rho_1+\rho_2)_{n_e} (1+n_e+g)_{N_e-n_e} (\frac12+\rho_1-r_2)_{n_e+n_p} (\frac12+\rho_2-r_2)_{n_e+n_p} }{ (N_e-n_e-n_p)! (\frac12+\rho_1+r_2)_{N_e-n_e} (\frac12+n_e+n_p+\rho_2-r_1)_{N_e-n_e-n_p} (1+n+g)_{n_e+n_p}^2 }.
\end{align}
The formulas for other values of $j$ and $\ell$ can be obtained by using the appropriate substitutions $\rho_1\leftrightarrow\rho_2$ and $r_1\leftrightarrow r_2$ in \eqref{Tcond1}--\eqref{Ortho-Norm}.

If $N$ is odd, it follows from~\eqref{Coeff-An} that the condition $u_{N+1}=0$ is equivalent to
one of the following restrictions:
\begin{gather} \label{Tcond2}
ii)~\rho_1+\rho_2=-\frac{N+1}{2},
\qquad
iii)~r_1+r_2=\frac{N+1}{2},
\qquad
iv)~\rho_1+\rho_2-r_1-r_2=-\frac{N+1}{2}.
\end{gather}
We refer to the possible truncation conditions as type $i)$ to type $iv)$. Note however that type $iv)$ leads to a~singularity in $u_n$ when $n=(N+1)/2$ and is therefore not admissible\footnote{It might be possible to absorb this singularity with some fine-tuning of the parameters as has been done for the Racah and $q$-Racah polynomials \cite{PR1,PR2} but this has not been explored yet and goes beyond the scope of this paper.}.
For type~$ii)$, the formulas~\eqref{gridi} and \eqref{Ortho-Weight} hold and the normalization factors are given by
\begin{align} \label{Ortho-Norm2}
 h_{n} = \frac{ n_e! N_e! (1+2\rho_1)_{N_e+1} (1\!-\!r_1\!-\!r_2)_{n_e} (1\!+\!n_e\!+\!g)_{N_e+1-n_e} (\frac12\!+\!\rho_1\!-\!r_1)_{n_e+n_p} (\frac12\!+\!\rho_1\!-\!r_2)_{n_e+n_p} }{ (N_e\!-\!n_e)! (\frac12+\rho_1+r_1)_{N_e+1-n_e-n_p} (\frac12+n_e+n_p+\rho_2-r_2)_{N_e+1-n_e-n_p} (1+n+g)_{n_e+n_p}^2 }.
\end{align}
For type $iii)$, the spectral points are
\begin{gather} \label{gridiii}
x_{k}=(-1)^{k}(r_1-k/2-1/4)-1/4,
\end{gather}
for $k=0,\dots,N$, the weight function is given by~\eqref{Ortho-Weight} with the substitutions
$(\rho_1,\rho_2,r_1,r_2)\rightarrow-(r_1,r_2,\rho_1,\rho_2)$ and the normalization factors read
{\small
\begin{align} \label{Ortho-Norm3}
 h_{n} = \frac{ n_e! N_e! (2r_2\!-\!N)_{N_e+1} (\rho_1\!+\!\rho_2\!-\!N_e)_{N_e+1+n_e} (\rho_1\!+\!\rho_2\!-\!N_e)_{n_e} (r_2\!+\!\rho_1\!-\!\tfrac12\!-\!N_e)_{n_e+n_p} (r_2\!+\!\rho_2\!-\!\tfrac12\!-\!N_e)_{n_e+n_p} }{ (N_e-n_e)! (\rho_1+\rho_2-N_e)_{n}^2 (r_2-\rho_1-\tfrac12-N_e)_{N_e+1-n_e-n_p} (r_2-\rho_2+\tfrac12+n_e+n_p)_{N_e+1-n_e-n_p} }.
\end{align}
}
To construct a bivariate extension of the Bannai-Ito polynomials, the different truncation conditions for different parities of $N$ will play an important role. 

The BI polynomials also verify a difference equation :
\begin{align} \label{DiffEqn1}
 \mathcal{L} B_n(x) = \lambda_n B_n(x)
\end{align}
with
\begin{align} \label{DiffEqn2}
 \mathcal{L} = \frac{(x-\rho_1)(x-\rho_2)}{2x} (1-R_x) + \frac{(x-r_1+\tfrac12)(x-r_2+\tfrac12)}{2x+1}(T_x^{1}R_x-1)
\end{align}
where $1$ is the identity operator, $R_x f(x)=f(-x)$ denotes the reflection operator and $T_{x}^{m}f(x)=f(x+m)$ is a shift operator. The eigenvalues are given by
\begin{align} \label{DiffEqn3}
 \lambda_n = \begin{cases}
              \frac{n}{2} \quad &\text{$n$ even,} \\
	      r_1+r_2-\rho_1-\rho_2-\frac{n+1}{2} \quad &\text{$n$ odd}.
             \end{cases}
\end{align}
It was shown in \cite{DunklShift} that $\mathcal{L}$ is in fact the most general first order Dunkl difference operator with orthogonal polynomials as eigenfunctions.

%----------------------------------------------------------------------------------------
\section{Limit from the bivariate q-Racah polynomials}

\subsection{Defining the bivariate Bannai-Ito polynomials}

\begin{definition}
The bivariate Bannai-Ito polynomials are defined by
\begin{align} \label{bivBI}
 B_{n_1,n_2}(z_1,z_2) = B_{n_1} \left(z_1-\tfrac14 ; \rho_1^{(1)}, \rho_2^{(1)}, r_1^{(1)}, r_2^{(1)} \right)B_{n_2}\left((-1)^{n_1}z_2-\tfrac14; \rho_1^{(2)}, \rho_2^{(2)}, r_1^{(2)}, r_2^{(2)}  \right)
\end{align}
where the $B_n$ and $\eta_n$ are as in \eqref{hypergeo-expr} and \eqref{eta} and the parameters are given by 
\begin{align} \label{paramB1}
\begin{aligned}
     \rho_1^{(1)} &= c-p_1+\tfrac12, \qquad &&r_1^{(1)} = \tfrac12 -p_1,  \\
     \rho_2^{(1)} &= z_2+p_2-\tfrac14,       \qquad &&r_2^{(1)} = z_2-p_2+\tfrac14  
\end{aligned}
\end{align}
and
\begin{align} \label{paramB2}
\begin{aligned}
     \rho_1^{(2)} &= \tfrac{n_1+1}{2}+c+p_2-p_1, \qquad &&r_1^{(2)} = \tfrac{1-n_1}{2}-p_1-p_2, \\
     \rho_2^{(2)} &= p_3 - (-1)^{n_1+N}(\tfrac{N}{2}+p_1+p_2+p_3), \qquad &&r_2^{(2)}= -p_3 - (-1)^{n_1+N}(\tfrac{N}{2}+p_1+p_2+p_3).
\end{aligned}
\end{align}
\end{definition}

It is useful to denote the first and second BI polynomial in the definition by $B_{n_1}^{(1)}(z_1)$ and $B_{n_2}^{(2)}(z_2)$.
Note that $B_{n_1}^{(1)}(z_1)$ contains the variable $z_2$, while $B_{n_2}^{(2)}(z_2)$ contains the degree $n_1$ in their respective parameters. We shall see that the $B_{n_1,n_2}(z_1,z_2)$ are orthogonal polynomials of degree $n_1+n_2 \le N$ in the variables $z_1$ and $z_2$ which depend on four parameters $p_1, p_2, p_3$ and $c$. 

Let us motivate this definition. In a spirit similar to the one that led to the discovery of the Bannai-Ito polynomials, we look at a $q\to-1$ limit of $q$-Racah polynomials in two-variable introduced by Gasper and Rahman in \cite{GR1} as a $q$-generalization of Tratnik's multivariable Racah polynomials. We start here by specializing their $q$-Racah polynomials to two variables. Consider the product of $q$-Racah polynomials $R^{(1)}_{n_1} \times R^{(2)}_{n_2}$ depending on four parameters $a_1, a_2, a_3$ and $b$ where $R^{(1)}_{n_1}$ and $R^{(2)}_{n_2}$ are defined by
\begin{align} \label{q-RacahOPs}
\begin{aligned}
 R^{(1)}_{n_1} &= \pGq{4}{3}{q^{-n_1}, ba_2q^{n_1}, q^{-x_1}, a_1 q^{x_1}}{bq, a_1a_2q^{x_2}, q^{-x_2}}{q;q} , \\[1em]
 R^{(2)}_{n_2} &= \pGq{4}{3}{q^{-n_2}, ba_2a_3q^{2n_1+n_2}, q^{n_1-x_2}, a_1a_2q^{n_1+x_2}}{ba_2q^{2n_1+1}, a_1a_2a_3q^{N+n_1}, q^{n_1-N}}{q;q} 
\end{aligned}
\end{align}
in terms of the usual basic hypergeometric function $_r \phi_s$ (see e.g. \cite{koekoek}). Up to normalization of the polynomials, those corresponds to the $q$-Racah polynomials introduced by Gasper and Rahman \cite{GR1}. Our goal is to take a $q\to-1$ limit of these polynomials in such a way that each generalized $q$-hypergeometric function reduces to a Bannai-Ito polynomials and that each parameter survives the limit. Let us first write the hypergeometric functions as series : 
\begin{align} \label{Rseries}
 R^{(1)}_{n_1} = \sum_{k=0}^{\infty} A^{(1)}_{k} q^k, \qquad \quad
 R^{(2)}_{n_2} = \sum_{k=0}^{\infty} A^{(2)}_{k} q^k
\end{align}
where the coefficients are given by
\begin{align} \label{A12coefs}
\begin{aligned}
 A^{(1)}_{k} = &\prod_{i=0}^{k-1} \frac{(1-q^{-n_1+i})(1-ba_2q^{n_1+i})(1-q^{-x_1+i})(1-a_1q^{x_1+i})}{(1-q^{1+i})(1-bq^{1+i})(1-a_1a_2q^{x_2+i})(1-q^{-x_2+i})} , \\
 A^{(2)}_{k} = &\prod_{i=0}^{k-1} \frac{(1-q^{-n_2+i})(1-ba_2a_3q^{2n_1+n_2+i})(1-q^{n_1-x_2+i})(1-a_1a_2q^{n_1+x_2+i})}{(1-q^{1+i})(1-ba_2q^{2n_1+1+i})(1-a_1a_2a_3q^{N+n_1+i})(1-q^{n_1-N+i})} .
\end{aligned}
\end{align}
Note that the coefficients of an hypergeometric series are usually written in terms of Pochhammer symbols, but for our purpose, it is essential to expand them as products because the parity of the dummy index $i$ will play an important role.
Now, achieve the $q\to-1$ limit with the following parametrization 
\begin{align} \label{parametrization}
\begin{aligned}
 &q\to -e^t, \qquad \quad t\to0, \qquad q^{x_1} \to (-1)^{\frac{s_1}{2}}e^{t y_1}, \quad q^{x_2} \to (-1)^{\frac{s_2}{2}}e^{t y_2}, \\
 &a_1\to (-1)^{\frac{s_3}{2}}e^{t\alpha_1}, \quad a_2\to (-1)^{\frac{s_4}{2}}e^{t\alpha_2}, \quad a_3\to (-1)^{\frac{s_5}{2}}e^{t\alpha_3}, \quad b\to (-1)^{\frac{s_6}{2}}e^{t\beta}
\end{aligned}
\end{align}
where the $s_i \in \{0,1\},\ i=1,2,\dots,6$ are integers to be determined. Let us now sketch how one chooses a proper set of $s_i$. First, insert the parametrization \eqref{parametrization} in the coefficients \eqref{A12coefs} to obtain
\begin{align} \label{dummycoef}
\begin{aligned}
 A^{(1)}_{k} = &\prod_{i=0}^{k-1} \frac{(1-(-1)^{b_1+i} e^{t B_1})(1-(-1)^{b_2+i} e^{t B_2})(1-(-1)^{b_3+i} e^{t B_3})(1-(-1)^{b_4+i} e^{t B_4})}{(1-(-1)^{b_5+i} e^{t B_5})(1-(-1)^{b_6+i} e^{t B_6})(1-(-1)^{b_7+i} e^{t B_7})(1-(-1)^{b_8+i} e^{t B_8})}, \\
 A^{(2)}_{k} = &\prod_{i=0}^{k-1} \frac{(1-(-1)^{b_9+i} e^{t B_9})(1-(-1)^{b_{10}+i} e^{t B_{10}})(1-(-1)^{b_{11}+i} e^{t B_{11}})(1-(-1)^{b_{12}+i} e^{t B_{12}})}{(1-(-1)^{b_{13}+i} e^{t B_{13}})(1-(-1)^{b_{14}+i} e^{t B_{14}})(1-(-1)^{b_{15}+i} e^{t B_{15}})(1-(-1)^{b_{16}+i} e^{t B_{16}})}
\end{aligned}
\end{align}
where the $b_j$ are linear combinations of $n_1, n_2, N$ and the $s_i$, while the $B_j$ are linear combinations of the dummy index $i$, the degrees $n_1, n_2, N$, the variables $y_1, y_2$ and the parameters $\alpha_1, \alpha_2, \alpha_3, \beta$. The choice of $s_i$ should be such that the $b_j$ are integers. Thus, depending on the parity of the $b_j$, each factor in the limit $t\to0$ will alternate between 0 and 2 for incrementing values of $i$. It is straightforward to see that the parities of the $b_j$ must be chosen in such a way that there are the same number of zeroes and twos in the numerator and in the denominator. Otherwise, the limit would diverge or become zero. 
Now, ratios of 2 will simply cancel out while ratios of 0 will give a non-trivial limit :
\begin{align}
 \frac{1+e^{tB_l}}{1+e^{tB_m}} \xrightarrow[]{t\to0\ } 1, \qquad \quad \frac{1-e^{tB_l}}{1-e^{tB_m}} \xrightarrow[]{t\to0\ } \frac{B_l}{B_m}.
\end{align}
The key to obtaining Bannai-Ito polynomials in the limit $t\to0$ is to chose the $s_i$ in such a way that, for each $i$, there are always 2 zeroes and 2 twos in both numerators and denominators. Such a choice is not unique, but it can be verified by enumeration that all possible choices yield results that are equivalent under affine transformations of the parameters. This implies that \eqref{dummycoef} will reduce to 
\begin{align} \label{dummycoef2}
\begin{aligned}
 A^{(1)}_{k} \xrightarrow[]{t\to0\ } \prod_{i\ \text{even}} \frac{B_{j_1}B_{j_2}}{B_{j_5}B_{j_6}} \times \prod_{i\ \text{odd}} \frac{B_{j_3}B_{j_4}}{B_{j_7}B_{j_8}}, \\
 A^{(2)}_{k} \xrightarrow[]{t\to0\ } \prod_{i\ \text{even}} \frac{B_{j_9}B_{j_{10}}}{B_{j_{13}}B_{j_{14}}} \times \prod_{i\ \text{odd}} \frac{B_{j_{11}}B_{j_{12}}}{B_{j_{15}}B_{j_{16}}}
\end{aligned} 
\end{align}
for some permutation $\pi\in S_{16}$ of the integers $j_k=\pi(k)$, $k=1,\dots,16$\ depending on the choice of the $s_i$. 
The explicit computation of the limit requires to consider separately all possible parities of the degrees $n_1, n_2, N$ and also of the dummy indices $k$ and $i$. Some notable features arise :
First, the products in \eqref{dummycoef2} can be written in terms of Pochhammer symbols. However, the products over even values of $i$ will get additional factors when $k$ is odd.
Thus, the sums in \eqref{Rseries} must be split between even and odd values of $k$. Each sum can be expressed as an hypergeometric $_4F_3$, but the additional factors for $k$ odd have to be pulled in front. One obtains a linear combination of two similar $_4F_3$ with some $+1$ shifts. It is then possible to compare the result with \eqref{hypergeo-expr} to express the result in terms of Bannai-Ito polynomials. 

Finally, let us consider without loss of generality one possible parametrization for the limit $q\to-1$ :
\begin{align} \label{param}
\begin{aligned}
 &q\to -e^t, \qquad \quad t\to0, \qquad q^{x_1} \to e^{t y_1}, \quad q^{x_2} \to e^{t y_2}, \\
 &a_1\to -e^{t\alpha_1}, \quad a_2\to e^{t\alpha_2}, \quad a_3\to e^{t\alpha_3}, \quad b\to -e^{t\beta}.
\end{aligned}
\end{align}
For convenience, we also use a different set of parameters :
\begin{align} \label{reparam}
\begin{aligned}
 \alpha_1 = 4p_1-1, \quad \alpha_2 = 4p_2, \quad \alpha_3 = 4p_3+1, \quad \beta = 2c,\quad y_1 = \tfrac12-2z_1, \quad y_2 = \tfrac12-2z_2-2p_1-2p_2.
\end{aligned}
\end{align}

Using \eqref{param} and \eqref{reparam}, a straightforward computation yields 
\begin{align*}
 R^{(1)}_{n_1} \to \ \frac{1}{\eta_{n_1}} B_{n_1} \left(z_1-\tfrac14 ; \rho_1^{(1)}, \rho_2^{(1)}, r_1^{(1)}, r_2^{(1)} \right) \qquad
 R^{(2)}_{n_2} \to \ \frac{1}{\eta_{n_2}} B_{n_2}\left((-1)^{n_1}z_2-\tfrac14; \rho_1^{(2)}, \rho_2^{(2)}, r_1^{(2)}, r_2^{(2)}  \right)
\end{align*}
where the parameters are given by \eqref{paramB1} and \eqref{paramB2}. Omitting the normalization factors, this corresponds to Definition 1 of the bivariate Bannai-Ito polynomials given above. There are two reasons for removing the factors $\eta_{n_i}$ : It is more natural to define the bivariate polynomials as a product of two monic Bannai-Ito polynomials and more importantly, the normalization factor $\eta_{n_1}$ being a rational function in $z_2$ would break the polynomial structure.  

\subsection{Truncation conditions and orthogonality relation} 

Given our definition of the bivariate Bannai-Ito polynomials, the most important property to verify is orthogonality. We begin by stating the result. 

\begin{proposition}
The bivariate Bannai-Ito polynomials defined in \eqref{bivBI} satisfy the orthogonality relation
\begin{align} \label{ORbivBI}
 \sum_{s=0}^{N} \sum_{r=0}^{N} w^{(1)}_{r,N-s}w^{(2)}_{s,N} B_{n_1,n_2}(z_1(r),z_2(s))B_{m_1,m_2}(z_1(r),z_2(s)) = H_{n_1,n_2,N} \delta_{n_1,m_1}\delta_{n_2,m_2}
\end{align}
where the grids are given by
\begin{align} 
z_1(r) &= \tfrac12 \left[ (-1)^{r+s+N}(r+s-N-2p_1+\tfrac12) \right]  \qquad r=0,\dots, N  \label{gridz1}\\
z_2(s) &= \tfrac12 \left[ (-1)^{s+N}(s-N-2p_1-2p_2+\tfrac12) \right] \qquad s=0,\dots, N  \label{gridz2}
\end{align}
the weights by 
\begin{align} \label{w1}
\begin{aligned}
 w^{(1)}_{2r,2\tilde{s}} &= \frac{(2p_2)_{r} (-\tilde{s})_{r} (1-2p_1-2\tilde{s})_{r} (\frac32+c-2p_1-\tilde{s})_{r}}{r!(\frac12-c-\tilde{s})_{r} (1-2p_1-\tilde{s})_{r} (1-2p_1-2p_2-2\tilde{s})_{r}}  \\[0.5em]
 w^{(1)}_{2r+1,2\tilde{s}} &= -\frac{(2p_2)_{r+1} (-\tilde{s})_{r+1} (1-2p_1-2\tilde{s})_{r} (\frac32+c-2p_1-\tilde{s})_{r}}{r!(\frac12-c-\tilde{s})_{r} (1-2p_1-\tilde{s})_{r+1} (1-2p_1-2p_2-2\tilde{s})_{r+1}}  \\[0.5em]
 w^{(1)}_{2r,2\tilde{s}+1} &= \frac{(2p_2)_{r} (-\tilde{s})_{r} (-2p_1-2\tilde{s})_{r} (\frac12+c-2p_1-\tilde{s})_{r}}{r!(-\frac12-c-\tilde{s})_{r} (1-2p_1-\tilde{s})_{r} (-2p_1-2p_2-2\tilde{s})_{r}}  \\[0.5em]
 w^{(1)}_{2r+1,2\tilde{s}+1} &= -\frac{(2p_2)_{r+1} (-\tilde{s})_{r} (-2p_1-2\tilde{s})_{r} (\frac12+c-2p_1-\tilde{s})_{r+1}}{r!(-\frac12-c-\tilde{s})_{r+1} (1-2p_1-\tilde{s})_{r} (-2p_1-2p_2-2\tilde{s})_{r+1}}
\end{aligned}
\end{align}
and 
{\small
\begin{align} \label{w2}
\begin{aligned}
 w^{(2)}_{2s,2N}     &= \frac{ (-1)^s \binom{N}{s} (\frac32+c-2p_1-N)_{s} (1-2p_1-2p_2-2N)_{s} (\frac12+2p_3)_{s}        }{ N!  (\frac12\!-\!c\!-\!N)_{s} (1\!-\!2p_1\!-\!2p_2\!-\!2N)_{2s} (\frac12\!-\!2p_1\!-\!2p_2\!-\!2p_3\!-\!2N)_{s} (1\!-\!2p_1\!-\!2N\!+\!2s)_{N-s} } \\[0.5em]
 w^{(2)}_{2s+1,2N}   &= \frac{ (-1)^s \binom{N-1}{s}(\frac32+c-2p_1-N)_{s} (1-2p_1-2p_2-2N)_{s} (\frac12+2p_3)_{s+1}      }{ (N\!-\!1)!  (\frac12\!-\!c\!-\!N)_{s} (1\!-\!2p_1\!-\!2p_2\!-\!2N)_{2s+1} (\frac12\!-\!2p_1\!-\!2p_2\!-\!2p_3\!-\!2N)_{s+1} (2\!-\!2p_1\!-\!2N\!+\!2s)_{N-s} } \\[0.5em]
 w^{(2)}_{2s,2N+1}   &= \frac{ (-1)^s \binom{N}{s}(\frac12+c-2p_1-N)_{s} (-2p_1-2p_2-2N)_{s} (\frac12+2p_3)_{s}         }{ N! (-\frac12\!-\!c\!-\!N)_{s} (-2p_1\!-\!2p_2\!-\!2N)_{2s} (-\frac12\!-\!2p_1\!-\!2p_2\!-\!2p_3\!-\!2N)_{s} (-2p_1\!-\!2N\!+\!2s)_{N+1-s} } \\[0.5em]
 w^{(2)}_{2s+1,2N+1} &= \frac{ (-1)^{s+1} \binom{N}{s} (\frac12+c-2p_1-N)_{s+1} (-2p_1-2p_2-2N)_{s} (\frac12+2p_3)_{s+1} }{ N! (-\frac12\!-\!c\!-\!N)_{s+1} (-2p_1\!-\!2p_2\!-\!2N)_{2s+1} (-\frac12\!-\!2p_1\!-\!2p_2\!-\!2p_3\!-\!2N)_{s+1} (1\!-\!2p_1\!-\!2N\!+\!2s)_{N-s} } 
\end{aligned}
\end{align}
}
and the normalization coefficients $H_{n_1,n_2,N}$ are given in the appendix.  
\end{proposition}

\begin{proof}
Our main tool will be the orthogonality relation of the univariate BI polynomials which depends on the truncation conditions \eqref{Tcond1} and \eqref{Tcond2}. 

Notice that definition \eqref{bivBI} involves a truncation parameter $N$ inherited from the $q$-Racah polynomials in the limit process and which appears in the parameters \eqref{paramB2}. This implies that the definition comes with truncation conditions that are already built-in. Indeed, one can easily check that $B_{n_2}^{(2)}(z_2)$ satisfies a mixture of type $i)$ \eqref{Tcond1} and type $iii)$ \eqref{Tcond2} truncation conditions : 
\begin{align} \label{truncB2}
 \begin{cases}
  r_1^{(2)}-\rho_2^{(2)} = \frac{N-n_1+1}{2} \qquad \text{if $N-n_1$ even,} \\[0.5em]
  r_1^{(2)}+r_2^{(2)} = \frac{N-n_1+1}{2} \qquad \text{if $N-n_1$ odd} 
 \end{cases}
\end{align}
with truncation parameters $N-n_1$. Both conditions impose grid points for the variable $z_2$ :
\begin{align} 
 (-1)^{n_1} z_2-\tfrac14 = x_s
\end{align}
where $x_s$ is given by \eqref{gridi} when $N+n_1$ is even and by \eqref{gridiii} when $N+n_1$ is odd. Inspecting this equation for each possible parities of $n_1$ and $N$, one obtains \eqref{gridz2}. Substituting this relation for $z_2$ in the parameters \eqref{paramB1}, one can now check that the truncation conditions satisfied by $B_{n_1}$ are
\begin{align} \label{truncB1}
 \begin{cases}
 r_1^{(1)}-\rho_2^{(1)} = \frac{N-s+1}{2} \qquad \text{if $N-s$ even,} \\[0.5em]
 r_1^{(1)}+r_2^{(1)}  = \frac{N-s+1}{2} \qquad \text{if $N-s$ odd.} 
 \end{cases}
\end{align}
Thus, the polynomial $B_{n_1}$ also satisfies the mixed type $i)$ and type $iii)$ truncation conditions with parameter $N-s$. Again, both conditions impose grid points for the variable $z_1$ :
\begin{align}
 z_1-\tfrac14 = x_r
\end{align}
with $x_r$ given by \eqref{gridi} for $N-s$ even and \eqref{gridiii} for $N-s$ odd. This amounts to \eqref{gridz1}. 
Recall that we are looking for an orthogonality relation of the form
\begin{align} \label{ORwish}
 \sum_{r,s} w^{(1)}_{r,N-s}w^{(2)}_{s,N} B_{n_1,n_2}(z_1(r),z_2(s))B_{m_1,m_2}(z_1(r),z_2(s)) = H_{n_1,n_2,N} \delta_{n_1,m_1}\delta_{n_2,m_2}.
\end{align}
In view of \eqref{truncB1}, the polynomial $B_{n_1}^{(1)}(z_1)$ satisfies the orthogonality relation
\begin{align} \label{ORB1}
 \sum_{r=0}^{N-s} w^{(1)}_{r,N-s} B_{n_1}^{(1)}(z_1(r))B_{m_1}^{(1)}(z_1(r)) = h_{n_1,N-s}^{(1)} \delta_{n_1,m_1}.
\end{align}
Using \eqref{Ortho-Weight}, one obtains the weights \eqref{w1}. The normalization coefficients are retrieved from \eqref{Ortho-Norm} and \eqref{Ortho-Norm3} :
{\small
\begin{align}
\begin{aligned}
  h_{2n_1,2\tilde{s}}^{(1)}     &= \frac{ \tilde{s}!~n_1!  (2p_2)_{n_1} (\frac12\!+\!c\!+\!n_1\!+\!2p_2)_{\tilde{s}-n_1} (\frac12\!+\!c\!+\!2p_2\!+\!\tilde{s})_{n_1} (\frac32\!+\!c\!-\!2p_1\!-\!\tilde{s})_{n_1} (1\!-\!2p_1\!-\!2\tilde{s})_{\tilde{s}} }{ (\tilde{s}-n_1)! (\frac12+c+n_1)_{\tilde{s}-n_1} (\frac12+c+n_1+2p_2)_{n_1}^2 (1-2p_1-2p_2-2\tilde{s})_{\tilde{s}-n_1} }, \\[0.5em]
  h_{2n_1+1,2\tilde{s}}^{(1)}   &= \frac{ \tilde{s}!~n_1! (2p_2)_{n_1+1} (\frac12\!+\!c\!+\!n_1\!+\!2p_2)_{\tilde{s}-n_1} (\frac12\!+\!c\!+\!2p_2\!+\!\tilde{s})_{n_1+1} (\frac32\!+\!c\!-\!2p_1\!-\!\tilde{s})_{n_1} (1\!-\!2p_1\!-\!2\tilde{s})_{\tilde{s}} }{ (\tilde{s}-n_1-1)! (\frac12+c+n_1+1)_{\tilde{s}-n_1-1} (\frac12+c+n_1+2p_2)_{n_1+1}^2 (1-2p_1-2p_2-2\tilde{s})_{\tilde{s}-n_1} }, \\[0.5em]
  h_{2n_1,2\tilde{s}+1}^{(1)}   &= \frac{ \tilde{s}!~n_1! (2p_2)_{n_1} (\frac12\!+\!c\!+\!n_1\!+\!2p_2)_{\tilde{s}+1-n_1} (\frac32\!+\!c\!+\!2p_2\!+\!\tilde{s})_{n_1} (\frac12\!+\!c\!-\!2p_1\!-\!\tilde{s})_{n_1} (-2p_1\!-\!2\tilde{s})_{\tilde{s}+1} }{ (\tilde{s}-n_1)! (\frac12+c+n_1)_{\tilde{s}+1-n_1} (\frac12+c+n_1+2p_2)_{n_1}^2 (-2p_1-2p_2-2\tilde{s})_{\tilde{s}+1-n_1} }, \\[0.5em]
  h_{2n_1+1,2\tilde{s}+1}^{(1)} &= \frac{\tilde{s}!~n_1! (2p_2)_{n_1+1} (\frac12\!+\!c\!+\!n_1\!+\!2p_2)_{\tilde{s}+1-n_1} (\frac12\!+\!c\!+\!2p_2\!+\!\tilde{s})_{n_1+1} (\frac12\!+\!c\!-\!2p_1\!-\!\tilde{s})_{n_1} (-2p_1\!-\!2\tilde{s})_{\tilde{s}+1} }{ (\tilde{s}-n_1)! (\frac32+c+n_1)_{\tilde{s}-n_1} (\frac12+c+n_1+2p_2)_{n_1+1}^2 (-2p_1-2p_2-2\tilde{s})_{\tilde{s}-n_1} }.
\end{aligned}
\end{align}
}
Using \eqref{ORB1} in \eqref{ORwish}, one gets that 
\begin{align}
 \sum_s h_{n_1,N-s}^{(1)} w^{(2)}_{s,N} B_{n_2}^{(2)}(z_2(s))B_{m_2}^{(2)}(z_2(s)) = H_{n_1,n_2,N} \delta_{n_2,m_2}
\end{align}
should be the orthogonality relation satisfied by the univariate BI polynomials $B_{n_2}^{(2)}(z_2)$. Indeed, using \eqref{truncB2} and the corresponding weights from section 1, one readily checks that the $w^{(2)}_{s,N}$ are given by \eqref{w2} and the normalization coefficients are those provided in the appendix. 

Hence the bivariate Bannai-Ito polynomials obey the orthogonality relation
\begin{align}
 \sum_{s=0}^{N-n_1} \sum_{r=0}^{N-s} w^{(1)}_{r,N-s}w^{(2)}_{s,N} B_{n_1,n_2}(z_1(r),z_2(s))B_{m_1,m_2}(z_1(r),z_2(s)) = H_{n_1,n_2,N} \delta_{n_1,m_1}\delta_{n_2,m_2}
\end{align}
with the weights, the grids and the normalization coefficients given. It is not hard to verify that both sums can be extended from $0$ to $N$ without changing the results. Indeed, one can check that all the extra terms are in fact zero because of the weights.
\end{proof}

%----------------------------------------------------------------------------------------
\section{Limit from the bivariate Askey-Wilson polynomials}

In this section, a different definition for the bivariate Bannai-Ito polynomials via a $q\to-1$ limit of the Askey-Wilson polynomials is investigated. While very similar to the approach from $q$-Racah polynomials, the main difference lies in the fact that the Askey-Wilson polynomials do not have truncation conditions. Hence, no truncation parameter $N$ is carried through the limit and a definition for untruncated bivariate Bannai-Ito polynomials is obtained. This definition has the advantage that its multispectrality relations can be expressed in terms of operators acting directly on the variables instead of acting on the orthogonality grids. The connection between both approaches is established.

\subsection{Untruncated bivariate Bannai-Ito polynomials}

\begin{definition} The untruncated bivariate Bannai-Ito polynomials are defined by 
\begin{align} \label{bivBI2}
 B_{n_1,n_2}(z_1,z_2) =\ &B_{n_1} \left(z_1-\tfrac14 ;\ \beta, z_2+\epsilon-\tfrac14, \alpha, z_2-\epsilon+\tfrac14 \right) \\[0.5em] \notag
                      &\times B_{n_2}\left((-1)^{n_1}z_2-\tfrac14;\ \beta+\epsilon+\tfrac{n_1}{2}, (1-\pi_{n_1})\gamma+\pi_{n_1}\delta, \alpha-\epsilon-\tfrac{n_1}{2}, (\pi_{n_1}-1)\delta-\pi_{n_1}\gamma   \right)
\end{align}
in terms of the monic BI polynomials $B_n(x)$ and where 
\begin{align}
 \pi_{n} = \frac{1+(-1)^n}{2} = \begin{cases}
                                 1 \quad \text{$n$ even,} \\
				 0 \quad \text{$n$ odd,}
                                \end{cases}
\end{align}
is the indicator function of even numbers.
\end{definition}

Note that Definition 2 reduces to Definition 1 \eqref{bivBI} of section 2 if we let 
\begin{align}
\begin{aligned} \label{reducing2to1}
 &\alpha\to \tfrac12-p_1, \qquad \beta \to \tfrac12+c-p_1, \qquad \epsilon \to p_2, \\[0.5em]
 &\gamma = \begin{cases}
           -\tfrac{N}{2}-p_1-p_2 \quad &\text{$N$ even,} \\
           \tfrac{N-1}{2}+p_1+p_2+p_3 \quad&\text{$N$ odd,}
          \end{cases}
\qquad
  \delta = \begin{cases}
            \tfrac{N-1}{2}+p_1+p_2+p_3 \quad &\text{$N$ even,} \\
            -\tfrac{N}{2}-p_1-p_2 \quad&\text{$N$ odd.}
          \end{cases}
\end{aligned}
\end{align}

To motivate Definition 2, let us first consider the Askey-Wilson polynomials 
\begin{align}
 \hat{p}_n(x;a,b,c,d) = \pGq{4}{3}{q^{-n},abcdq^{n-1},az,az^{-1}}{ab,ac,ad}{q;q}
\end{align}
in the variable $x=\tfrac12(z+z^{-1})$. The bivariate Askey-Wilson polynomials depending on five parameters $a,b,c,d,a_2$ as introduced by Gasper and Rahman in \cite{GR2} are 
\begin{align} \label{AWOPs}
\begin{aligned}
\hat{P}_{n_1,n_2}(x_1,x_2) = \hat{p}_{n_1}(x_1;a,b,a_2z_2,a_2z_2^{-1})\hat{p}_{n_2}(x_2;aa_2q^{n_1},ba_2q^{n_1},c,d).
%
% P^{(1)}_{n_1}(\cos\theta_1) &= \hat{p}_{n_1}(x_1;a,b,a_2z_2,a_2z_2^{-1})= \pGq{4}{3}{q^{-n_1}, aba_2^2q^{n_1-1}, a e^{i\theta_1}, a e^{-i\theta_1}}{ab, aa_2e^{i\theta_2}, a a_2 e^{-i\theta_2}}{q;q} , \\[1em]
% P^{(2)}_{n_2}(\cos\theta_2) &= \hat{p}_{n_2}(x_2;aa_2q^{n_1},ba_2q^{n_1},c,d)=\pGq{4}{3}{q^{-n_2}, abcda_2^2q^{2n_1+n_2-1}, aa_2q^{n_1}e^{i\theta_2},aa_2q^{n_1}e^{-i\theta_2}}{aba_2^2q^{2n_1}, aca_2q^{n_1}, ada_2q^{n_1}}{q;q} 
\end{aligned}
\end{align}
The limiting procedure will be the same as in the previous section. Briefly, the choice of parametrization amounts to a selection of phases in front of each parameter defined as exponentials. Expanding the hypergeometric functions as series and expanding the Pochhammer symbols as products, one obtains an expression of the form \eqref{dummycoef} and must select the phases in such a way that there are always 2 zeroes and 2 twos in both numerators and denominators for each value of the dummy index $i$. Again, this choice is not unique, but all possibilities can again be shown to yield equivalent Bannai-Ito polynomials under affine transformations of the parameters. 

We take
\begin{align} \label{paramAW}
\begin{aligned}
 &q\to -e^t, \quad t\to0, \qquad z_1 \to e^{ t y_1}, \quad z_2 \to e^{ t y_2}, \\
 &a\to ie^{t\tilde{a}}, \quad b\to -ie^{t\tilde{b}}, \quad c\to ie^{t\tilde{c}}, \quad d\to -ie^{t\tilde{d}}, \quad a_2\to e^{t \tilde{a_2}}
\end{aligned}
\end{align}
with the reparametrization
\begin{align} \label{reparamAW}
 y_1 = -2z_1, \quad y_2=-2z_2, \quad \tilde{a} = -2\alpha+\tfrac12, \quad \tilde{b} = 2\beta+\tfrac12,\quad 
 \tilde{c}=2\gamma+\tfrac12, \quad \tilde{d}=2\delta+\tfrac12, \quad \tilde{a_2}=2\epsilon.
\end{align}
This gives
\begin{align}
\begin{aligned}
 P^{(1)}_{n_1}(\cos\theta_1) \to \ &\frac{1}{\eta_{n_1}} B_{n_1} \left(z_1-\tfrac14 ;\ \beta, z_2+\epsilon-\tfrac14, \alpha, z_2-\epsilon+\tfrac14 \right) \\ 
 P^{(2)}_{n_2}(\cos\theta_2) \to \ &\frac{1}{\eta_{n_2}} B_{n_2}\left((-1)^{n_1}z_2-\tfrac14;\ \beta+\epsilon+\tfrac{n_1}{2}, (1\!-\!\pi_{n_1}\!)\gamma+\pi_{n_1}\delta, \alpha-\epsilon-\tfrac{n_1}{2}, (\pi_{n_1}\!\!-\!1)\delta-\pi_{n_1}\gamma   \right).
\end{aligned}
\end{align}
The definition for the untruncated bivariate Bannai-Ito polynomials is thus obtained by taking the product of the corresponding monic Bannai-Ito polynomials, again dropping the normalization factors.

\subsection{Multispectrality of the bivariate BI polynomials}

Iliev demonstrated the multispectrality of the bivariate Askey-Wilson polynomials \eqref{AWOPs} (with a different normalization) in \cite{Iliev}. This section examines how the multispectrality relations of these polynomials are carried in the $q\to-1$ limit. 
Recalling the definition of the shift operator $T^m_{x}$ and the reflection operator $R_x$ given after \eqref{DiffEqn2}, we have the following equations in the variables $z_1$ and $z_2$.
%\begin{align*}
% T^m_{x}f(x) = f(x+m) \qquad R_x f(x) = f(-x).
%\end{align*}

\begin{proposition}
  The untruncated bivariate Bannai-Ito polynomials \eqref{bivBI2} obey the difference equations
  \begin{align}
  L_1 B_{n_1,n_2}(z_1,z_2) &= \mu_{n_1} B_{n_1,n_2}(z_1,z_2)         \label{DiffEqnL1} \\
  L_2 B_{n_1,n_2}(z_1,z_2) &= \lambda_{n_1,n_2} B_{n_1,n_2}(z_1,z_2) \label{DiffEqnL2}
  \end{align}
  for the operators
  \begin{align} \label{L1}
  L_1 = \frac{(\epsilon-z_1+z_2)(z_1-\beta-\tfrac14)}{2(z_1-\tfrac14)} (T^{-1/2}_{z_1}R_{z_1}-1) + \frac{(\epsilon+z_1-z_2)(z_1-\alpha+\tfrac14)}{2(z_1+\tfrac14)} (T^{1/2}_{z_1}R_{z_1} - 1)
  \end{align}
  and 
  \begin{align} \label{L2}
  L_2 = \sum_{i,j=-1}^{1} c_{i,j} T_{z_1}^{i/2}R_{z_1}^{i}   T_{z_2}^{j/2}R_{z_2}^{j}
  \end{align}
  with coefficients
%   \begin{align} \label{L2coef}
%   \begin{aligned}
%   c_{-1,-1} &=\frac{(-4 \alpha +4 z_1+1) (4 \gamma +4 z_2+1) (2\epsilon+2 z_1+2 z_2+1)}{8 (4 z_1+1) (4 z_2+1)} \\[0.5em]
%   c_{-1,0}  &=\frac{ (-4 \alpha +4 z_1+1) (\epsilon+z_1-z_2) (\delta(4z_2+1)-\gamma(4z_2-1))}{(4 z_1+1) (4 z_2-1) (4 z_2+1)} \\[0.5em]
%   c_{-1,1}  &=\frac{(-4 \alpha +4 z_1+1) (-4 \delta +4 z_2-1) (\epsilon+z_1-z_2)}{4(4 z_1+1) (4 z_2-1)} \\[0.5em]
%   c_{0,-1}  &=\frac{ (4 \gamma +4 z_2+1) (\epsilon-z_1+z_2) (\alpha(4z_1-1)+\beta(4z_1+1))}{(4 z_1-1) (4 z_1+1) (4 z_2+1)} \\[0.5em]
%   c_{0,0}   &= \alpha (\gamma -\delta +\tfrac12)+\beta(\gamma-\delta-\tfrac12)-\tfrac12(\gamma+\delta+\tfrac12)-\epsilon \\[0.5em] & \qquad +\frac{ (4 \epsilon+16 z_1 z_2-1) (\alpha(4z_1-1)+\beta(4z_1+1)) (\delta(4z_2+1)-\gamma(4z_2-1))}{(4 z_1-1) (4 z_1+1) (4 z_2-1) (4 z_2+1)} \\[0.5em]   
%   c_{0,1}   &=\frac{ (-4 \delta +4 z_2-1) (\epsilon+z_1-z_2) (\alpha(4z_1-1)+\beta(4z_1+1))}{(4 z_1-1) (4 z_1+1) (4 z_2-1)} \\[0.5em]
%   c_{1,-1}  &=\frac{(-4 \beta +4 z_1-1) (4 \gamma +4 z_2+1) (\epsilon-z_1+z_2)}{4(4 z_1-1) (4 z_2+1)} \\[0.5em]
%   c_{1,0}   &=\frac{ (-4 \beta +4 z_1-1) (\epsilon-z_1+z_2) (\delta(4z_2+1)-\gamma(4z_2-1))}{(4 z_1-1) (4 z_2-1) (4 z_2+1)} \\[0.5em]
%   c_{1,1}   &=\frac{(-4 \beta +4 z_1-1) (-4 \delta +4 z_2-1) (2 \epsilon-2 z_1-2 z_2+1)}{8 (4 z_1-1) (4 z_2-1)}.
%   \end{aligned}
%   \end{align}
  \begin{align} \label{L2coef}
  \begin{aligned}
  c_{-1,-1} &=\frac{(z_1-\alpha+\tfrac14) (z_2+\gamma+\tfrac14) (\epsilon+z_1+z_2+\tfrac12)}{4(z_1+\tfrac14)(z_2+\tfrac14)} \\[0.5em]
  c_{-1,0}  &=\frac{ (z_1-\alpha+\tfrac14) (\epsilon+z_1-z_2) (\delta(z_2+\tfrac14)-\gamma(z_2-\tfrac14))}{4(z_1+\tfrac14) (z_2-\tfrac14) (z_2+\tfrac14)} \\[0.5em]
  c_{-1,1}  &=\frac{(z_1-\alpha+\tfrac14) (z_2-\delta-\tfrac14) (\epsilon+z_1-z_2)}{4(z_1+\tfrac14) (z_2-\tfrac14)} \\[0.5em]
  c_{0,-1}  &=\frac{ (z_2+\gamma+\tfrac14) (\epsilon-z_1+z_2) (\alpha(z_1-\tfrac14)+\beta(z_1+\tfrac14))}{4(z_1-\tfrac14) (z_1+\tfrac14) (z_2+\tfrac14)} \\[0.5em]
  c_{0,0}   &= \alpha (\gamma -\delta +\tfrac12)+\beta(\gamma-\delta-\tfrac12)-\tfrac12(\gamma+\delta+\tfrac12)-\epsilon \\[0.5em] & \qquad +\frac{ (\epsilon+4z_1 z_2-\tfrac14) (\alpha(z_1-\tfrac14)+\beta(z_1+\tfrac14)) (\delta(z_2+\tfrac14)-\gamma(z_2-\tfrac14))}{4(z_1-\tfrac14) (z_1+\tfrac14) (z_2-\tfrac14) (z_2+\tfrac14)} \\[0.5em]   
  c_{0,1}   &=\frac{ (z_2-\delta-\tfrac14) (\epsilon+z_1-z_2) (\alpha(z_1-\tfrac14)+\beta(z_1+\tfrac14))}{4(z_1-\tfrac14) (z_1+\tfrac14) (z_2-\tfrac14)} \\[0.5em]
  c_{1,-1}  &=\frac{(z_1-\beta-\tfrac14) (z_2+\gamma +\tfrac14) (\epsilon-z_1+z_2)}{4(z_1-\tfrac14) (z_2+\tfrac14)} \\[0.5em]
  c_{1,0}   &=\frac{ (z_1-\beta-\tfrac14) (\epsilon-z_1+z_2) (\delta(z_2+\tfrac14)-\gamma(z_2-\tfrac14))}{4(z_1-\tfrac14) (z_2-\tfrac14) (z_2+\tfrac14)} \\[0.5em]
  c_{1,1}   &=\frac{(z_1-\beta-\tfrac14) (z_2-\delta-\tfrac14) (\epsilon-z_1-z_2+\tfrac12)}{4 (z_1-\tfrac14) (z_2-\tfrac14)}.
  \end{aligned}
  \end{align}
  The eigenvalues are given by
  \begin{align}
  \mu_{n_1} = \begin{cases}
		    \frac{n_1}{2} \quad&\text{$n_1$ even,} \\
		    -\frac{n_1}{2}+\alpha-\beta-2\epsilon \quad&\text{$n_1$ odd,}
		  \end{cases}
  \end{align}
  and
  \begin{align}
  \lambda_{n_1,n_2} =  \begin{cases}
				  \frac{n_1+n_2}{2} \quad&\text{$n_1+n_2$ even,} \\
				  \frac{n_1+n_2+1}{2} -\alpha+\beta+\gamma+\delta+2\epsilon \quad&\text{$n_1+n_2$ odd.}
			\end{cases}
  \end{align}
\end{proposition}
  
\begin{proof}
  Consider the renormalized Askey-Wilson polynomials 
  \begin{align}
    p_n(x;a,b,c,d) = \xi_{n}(a,b,c,d)\hat{p}_n(x;a,b,c,d)
  \end{align}
  where
  \begin{align} \label{xi}
  \xi_n(a,b,c,d) = \frac{(ab,ac,ad;q)_n}{a^n}
  \end{align}
  and their corresponding bivariate extension
  \begin{align} \label{normalizedAW}
  P_{n_1,n_2}(x_1,x_2) = \zeta_{n_1,n_2} p_{n_1}(x_1;a,b,a_2z_2,a_2z_2^{-1}) p_{n_2}(x_2;aa_2q^{n_1},ba_2q^{n_1},c,d)
  \end{align}
  with normalization
  \begin{align} \label{zeta}
  \zeta_{n_1,n_2} = \frac{c^{n_1+n_2}a_2^{n_1}}{(a_2^2;q)_{n_1} (aca_2;q)_{n_1+n_2} (bca_2;q)_{n_1+n_2} (cd;q)_{n_2}}.
  \end{align}
  They obey the $q$-difference equation \cite{Iliev}
  \begin{align}
  \mathcal{L} P_{n_1,n_2}(x_1,x_2) = \Lambda_{n_1,n_2} P_{n_1,n_2}(x_1,x_2)
  \end{align}
  where 
  \begin{align} \label{mathcalL}
  \mathcal{L} = \sum_{i,j=-1}^{1} C_{i,j} E_{q,z_1}^{i}E_{q,z_2}^{j}
  \end{align}
  in terms of the shifts operators $E_{q,z}$ which send $z\to qz$. The explicit expression for the coefficients $C_{i,j}$ can be found in the appendix and the eigenvalues are given by 
  \begin{align}
    \Lambda_{n_1,n_2} = \left(q^{-n_1-n_2}-1\right) \left(1-a a_2^2 b c d q^{n_1+n_2-1}\right).
  \end{align}
  The difference equation for the bivariate Bannai-Ito polynomials is found as a limit of this relation. The operator $\mathcal{L}$ will correspond to the operator $L_2$ given by \eqref{L2} in the limit \eqref{paramAW} with the reparametrization \eqref{reparamAW}. The coefficients $c_{i,j}$ are obtained by the limits
  \begin{align}
  c_{i,j} = \lim_{q\to-1}\ \frac{C_{i,j}}{4(1+q)}
  \end{align}
  and the eigenvalues by
  \begin{align}
  \lambda_{n_1,n_2} = \lim_{q\to-1} \frac{\Lambda_{n_1,n_2}}{4(1+q)} = \begin{cases}
								  \frac{n_1+n_2}{2} \quad&\text{$n_1+n_2$ even,} \\
								  \frac{n_1+n_2+1}{2} -\alpha+\beta+\gamma+\delta+2\epsilon \quad&\text{$n_1+n_2$ odd.}
								\end{cases}
  \end{align}
  The factor $\tfrac14$ is just for convenience. The bivariate BI polynomials \eqref{bivBI2} will thus satisfy the difference equation \eqref{DiffEqnL2}.

  The difference equation \eqref{DiffEqnL1} follows directly from the univariate Dunkl difference equation given by \eqref{DiffEqn1}, \eqref{DiffEqn2} and \eqref{DiffEqn3}. It is also possible to obtain it from a $q\to-1$ limit of the bivariate Askey-Wilson polynomials second $q$-difference equation \cite{Iliev}.
\end{proof}

Let us now turn to the recurrence relations. 

\begin{proposition}
The untruncated bivariate Bannai-Ito polynomials $B_{n_1,n_2}(z_1,z_2)$ defined in \eqref{bivBI2} verify the 3-term recurrence relation
\begin{align}
 \left((-1)^{n_1}z_2-\tfrac14\right) B_{n_1,n_2}(z_1,z_2) = B_{n_1,n_2+1}(z_1,z_2) &+ (\beta+\epsilon+\tfrac{n_1}{2}-A_{n_2}-C_{n_2}) B_{n_1,n_2}(z_1,z_2) \\ &+ A_{n_2-1}C_{n_2}B_{n_1,n_2-1}(z_1,z_2) \notag
\end{align}
where the coefficients $A_{n_2}$ and $C_{n_2}$ are given by \eqref{Coeff-An} with the parameters $\rho_1,\rho_2,r_1,r_2$ being those of the second BI polynomial of \eqref{bivBI2}.

They also satisfy the 9-term recurrence relation
\begin{align} \label{RelRecBI9}
\begin{aligned}
  (z_1-\alpha^2+\beta^2)B_{n_1,n_2}(z_1,z_2) = \quad &\theta^{(1)}_{n_1,n_2}B_{n_1+1,n_2\phantom{-1}} + \theta^{(2)}_{n_1,n_2}B_{n_1+1,n_2-1} +\theta^{(3)}_{n_1,n_2}B_{n_1+1,n_2-2} \\
	                                            +\ &\theta^{(4)}_{n_1,n_2}B_{n_1,n_2+1\phantom{-1}} +   \theta^{(5)}_{n_1,n_2}B_{n_1,n_2\phantom{-1}\phantom{-1}}     +\theta^{(6)}_{n_1,n_2}B_{n_1,n_2-1} \\
	                                            +\ &\theta^{(7)}_{n_1,n_2}B_{n_1-1,n_2+2} + \theta^{(8)}_{n_1,n_2}B_{n_1-1,n_2+1} +\theta^{(9)}_{n_1,n_2}B_{n_1-1,n_2}
\end{aligned}
\end{align}
where the explicit expression for the coefficients $\theta^{(i)}_{n_1,n_2}$ are given in the appendix. 
\end{proposition}

\begin{proof}
The polynomials $P_{n_1,n_2}(x_1,x_2)$ verify \cite{Iliev}
\begin{align} \label{RelRecAW9}
\begin{aligned}
 ca_2 \big[\tfrac{(a\!+\!b)(ab\!+\!q)}{ab(1\!+\!q)}\!-\!z_1\!-\!z_1^{-1}\big]P_{n_1,n_2}(x_1,x_2) =\ &\tau^{(1)}_{n_1,n_2}P_{n_1+1,n_2 \phantom{-1}} + \tau^{(2)}_{n_1,n_2}P_{n_1+1,n_2-1} +\tau^{(3)}_{n_1,n_2}P_{n_1+1,n_2-2} \\
	                                                      +\ &\tau^{(4)}_{n_1,n_2}P_{n_1,n_2+1\phantom{-1}} +   \tau^{(5)}_{n_1,n_2}P_{n_1,n_2\phantom{-1}\phantom{-1}}     +\tau^{(6)}_{n_1,n_2}P_{n_1,n_2-1} \\
	                                                      +\ &\tau^{(7)}_{n_1,n_2}P_{n_1-1,n_2+2} + \tau^{(8)}_{n_1,n_2}P_{n_1-1,n_2+1} +\tau^{(9)}_{n_1,n_2}P_{n_1-1,n_2}.
\end{aligned}
\end{align}
The expression for the coefficients $\tau^{(i)}_{n_1,n_2}$ can be found in the appendix.
This will become a 9-term recurrence relation for the bivariate Bannai-Ito polynomials in the $q\to-1$ limit. The only tricky part is to keep track of all the changes in normalization of the various polynomials in play. Denote by
\begin{align}
 \mathcal{N}_{n_1,n_2} = \zeta_{n_1,n_2} \xi_{n_1}(x_1;a,b,a_2z_2,a_2z_2^{-1}) \xi_{n_2}(x_2;aa_2q^{n_1},ba_2q^{n_1},c,d)
\end{align}
the normalization factors that appear in \eqref{normalizedAW} and by
\begin{align}
 \mathcal{M}_{n_1,n_2} = \eta_{n_1}^{(1)} \eta_{n_2}^{(2)}
\end{align}
the normalization coefficients in the monic BI OPs \eqref{bivBI2} given by \eqref{eta}. Now, the recurrence coefficients are obtained by the following limits :
\begin{align*}
\begin{aligned}
 \theta^{(1)}_{n_1,n_2} &= \frac{\mathcal{M}_{n_1,n_2}}{\mathcal{M}_{n_1+1,n_2\phantom{-1}}} \lim_{q\to-1} \frac{\mathcal{N}_{n_1+1,n_2}}{\mathcal{N}_{n_1,n_2}} \frac{\tau^{(1)}_{n_1,n_2}}{4(1\!+\!q)},     \qquad &&\theta^{(6)}_{n_1,n_2} = \frac{\mathcal{M}_{n_1,n_2}}{\mathcal{M}_{n_1,n_2-1\phantom{-1}}} \lim_{q\to-1} \frac{\mathcal{N}_{n_1,n_2-1}}{\mathcal{N}_{n_1,n_2}} \frac{\tau^{(6)}_{n_1,n_2}}{4(1\!+\!q)}, \\[0.5em]
 \theta^{(2)}_{n_1,n_2} &= \frac{\mathcal{M}_{n_1,n_2}}{\mathcal{M}_{n_1+1,n_2-1}} \lim_{q\to-1} \frac{\mathcal{N}_{n_1+1,n_2-1}}{\mathcal{N}_{n_1,n_2}} \frac{\tau^{(2)}_{n_1,n_2}}{4(1\!+\!q)},             \qquad &&\theta^{(7)}_{n_1,n_2} = \frac{\mathcal{M}_{n_1,n_2}}{\mathcal{M}_{n_1-1,n_2+2}} \lim_{q\to-1} \frac{\mathcal{N}_{n_1-1,n_2+2}}{\mathcal{N}_{n_1,n_2}} \frac{\tau^{(7)}_{n_1,n_2}}{4(1\!+\!q)}, \\[0.5em]
 \theta^{(3)}_{n_1,n_2} &= \frac{\mathcal{M}_{n_1,n_2}}{\mathcal{M}_{n_1+1,n_2-2}} \lim_{q\to-1} \frac{\mathcal{N}_{n_1+1,n_2-2}}{\mathcal{N}_{n_1,n_2}} \frac{\tau^{(3)}_{n_1,n_2}}{4(1\!+\!q)},             \qquad &&\theta^{(8)}_{n_1,n_2} = \frac{\mathcal{M}_{n_1,n_2}}{\mathcal{M}_{n_1-1,n_2+1}} \lim_{q\to-1} \frac{\mathcal{N}_{n_1-1,n_2+1}}{\mathcal{N}_{n_1,n_2}} \frac{\tau^{(8)}_{n_1,n_2}}{4(1\!+\!q)}, \\[0.5em]
 \theta^{(4)}_{n_1,n_2} &= \frac{\mathcal{M}_{n_1,n_2}}{\mathcal{M}_{n_1,n_2+1\phantom{-1}}} \lim_{q\to-1} \frac{\mathcal{N}_{n_1,n_2+1}}{\mathcal{N}_{n_1,n_2}} \frac{\tau^{(4)}_{n_1,n_2}}{4(1\!+\!q)},     \qquad &&\theta^{(9)}_{n_1,n_2} = \frac{\mathcal{M}_{n_1,n_2}}{\mathcal{M}_{n_1-1,n_2\phantom{-1}}} \lim_{q\to-1} \frac{\mathcal{N}_{n_1-1,n_2}}{\mathcal{N}_{n_1,n_2}} \frac{\tau^{(9)}_{n_1,n_2}}{4(1\!+\!q)},  \\[0.5em]
 \theta^{(5)}_{n_1,n_2} &=  \lim_{q\to-1} \frac{\tau^{(5)}_{n_1,n_2}}{4(1\!+\!q)}. 
\end{aligned}
\end{align*}
The limits are assumed to be parametrized by \eqref{paramAW} and \eqref{reparamAW}. The results of these limits can be found in the appendix. 
Moreover, the recurrence relation operator is obtained via 
\begin{align}
 \lim_{q\to-1} \frac{ca_2}{4(1+q)}\big[\tfrac{(a\!+\!b)(ab\!+\!q)}{ab(1\!+\!q)}\!-\!z_1\!-\!z_1^{-1}\big] = z_1-\alpha^2+\beta^2.
\end{align}
Combining all these results, the recurrence relation \eqref{RelRecAW9} reduces to the desired 9-term recurrence relation for the bivariate Bannai-Ito polynomials.

The 3-term recurrence relation simply follows from the the recurrence relation of the univariate Bannai-Ito polynomials \eqref{Recurrence-Relation} applied to the second polynomial of \eqref{bivBI2}.
\end{proof}

These two propositions establish the full multispectrality of the bivariate Bannai-Ito polynomials. Importantly, the recurrence relations prove that the $B_{n_1,n_2}$ are polynomials and not simply rational functions of $z_1$ and $z_2$.

%----------------------------------------------------------------------------------------
\section{Conclusion}

This paper has enlarged the catalogue of orthogonal polynomials in two variables with the construction of bivariate polynomials of Bannai-Ito type. Their identification and characterization made use of the $q\to-1$ limits of both the bivariate $q$-Racah and Askey-Wilson polynomials of Gasper and Rahman. The first instance led to a truncated version (Definition 1) equipped with a set of positive-definite weights on a two-dimensionnal lattice against which the BI polynomials are orthogonal. The $q\to-1$ limit of the bivariate Askey-Wilson polynomials yielded untruncated Bannai-Ito polynomials in two variables (Definition 2) out of which the finite ones (Definition 1) can be obtained by the choice of parameters \eqref{reducing2to1}. This latter approach allowed for the identification of the difference equations and recurrence relations obeyed by the resulting functions showing in particular that they are indeed polynomials. Let us remark that the finite bivariate Bannai-Ito polynomials that have been found do only make use of the truncation conditions $i)$ and $iii)$ that the univariate polynomials admit. The question of whether there are other bivariate extensions that rely on different reduction mixtures and in particular condition $ii)$ is open and certainly worth exploring.

In another vein, one may wonder if there are natural multivariate generalizations of the Bannai-Ito polynomials along the symmetric function direction. In this respect, the examination of the $q\to-1$ limit of the Koornwinder polynomials of $BC_2$ type could prove illuminating and is envisaged.

We have initiated this exploration of the Bannai-Ito polynomials in many variables within the Tratnik framework because of the expected occurence of extensions of that type in the representation theory of the higher rank Bannai-Ito algebra \cite{HigherRank} as well as in certain superintegrable models that have been constructed \cite{DunklS3,DunklSN}. Let us mention the following to be concrete. A Hamiltonian system on the 3-sphere whose symmetries realize the Bannai-Ito algebra of rank 2 has been constructed in \cite{DunklS3} and various bases of wavefunctions have been explicitly obtained using the Cauchy-Kovalevskaia extension theorem. It is expected that bivariate Bannai-Ito polynomials arise in the interbasis connection coefficients. Do these overlaps coincide with the two-variable polynomials constructed here or do they belong to another extension yet to be found? We plan on looking into this in the near future. Another related question is to determine the algebra underscoring the multispectrality of the two variable BI polynomials we have defined, that is the algebra generated by $L_1, L_2, x_1$ and $x_2$. How does the resulting algebra compare with the rank 2 Bannai-Ito algebra? We hope to report on most of these questions soon.

%----------------------------------------------------------------------------------------
\section*{Acknowledgments}
JML holds an Alexander-Graham-Bell PhD fellowship from the Natural Science and Engineering Research Council (NSERC) of Canada. LV is grateful to NSERC for support through a discovery grant.

%----------------------------------------------------------------------------------------
\appendix
\section{Appendix}
To make the article more reader friendly, some cumbersome formulas have been relegated to this appendix. 

The normalization coefficients appearing in the orthogonality relation \eqref{ORbivBI} are given by 
\begin{align*}
  \begin{aligned}
    H_{2n_1,2n_2,2N}       &= \frac{n_1!n_2! (2 p_2)_{n_1} (2 p_3\!+\!\frac{1}{2})_{n_2} (c\!+\!n_1\!+\!2 p_2\!+\!\frac{1}{2})_N (c\!-\!N\!-\!2 p_1\!+\!\frac{3}{2})_{n_1\!+\!n_2}        }{  (N\!-\!n_1\!-\!n_2)!  (c\!+\!n_1\!+\!\frac{1}{2})_{N\!-\!n_1} (c\!+\!n_1\!+\!2 p_2\!+\!\frac{1}{2})_{n_1}^2 (c\!+\!2 n_1\!+\!n_2\!+\!2 p_2\!+\!\frac{1}{2})_{N\!-\!n_1\!-\!n_2} } \\[0.5em]
			  &\quad\times \frac{(c\!+\!N\!+\!n_1\!+\!2 p_2\!+\!2 p_3\!+\!1)_{n_2} (c\!+\!2 n_1\!+\!n_2\!+\!2 p_2\!+\!2 p_3\!+\!1)_{N\!-\!n_1\!-\!n_2}    }{    (c\!+\!2 n_1\!+\!n_2\!+\!2 p_2\!+\!2 p_3\!+\!1)_{n_2}^2 (\!-\!2 N\!-\!2 p_1\!-\!2 p_2\!-\!2 p_3\!+\!\frac{1}{2})_{N\!-\!n_1\!-\!n_2}   }  
  \end{aligned}
\end{align*}
\begin{align*}
  \begin{aligned}
    H_{2n_1\!+\!1,2n_2,2N}     &= \frac{n_1!n_2! (2 p_2)_{n_1\!+\!1} (2 p_3\!+\!\frac{1}{2})_{n_2} (c\!+\!n_1\!+\!2 p_2\!+\!\frac{1}{2})_N (c\!-\!N\!-\!2 p_1\!+\!\frac{3}{2})_{n_1\!+\!n_2}      }{  (N\!-\!n_1\!-\!n_2\!-\!1)!(c\!+\!n_1\!+\!\frac{3}{2})_{N\!-\!n_1\!-\!1} (c\!+\!n_1\!+\!2 p_2\!+\!\frac{1}{2})_{n_1\!+\!1}^2  (c\!+\!2 n_1\!+\!n_2\!+\!2 p_2\!+\!\frac{3}{2})_{N\!-\!n_1\!-\!n_2\!-\!1}   } \\[0.5em]
			      &\quad\times \frac{  (c\!+\!N\!+\!n_1\!+\!2 p_2\!+\!2 p_3\!+\!2)_{n_2} (c\!+\!2 n_1\!+\!n_2\!+\!2 p_2\!+\!2 p_3\!+\!2)_{N\!-\!n_1\!-\!n_2}    }{   (c\!+\!2 n_1\!+\!n_2\!+\!2 p_2\!+\!2 p_3\!+\!2)_{n_2}^2 (\!-\!2 N\!-\!2 p_1\!-\!2 p_2\!-\!2 p_3\!+\!\frac{1}{2})_{N\!-\!n_1\!-\!n_2}   }  
  \end{aligned}
\end{align*}
\begin{align*}
  \begin{aligned}
    H_{2n_1,2n_2\!+\!1,2N}    &= \frac{n_1!n_2! (2 p_2)_{n_1} (2 p_3\!+\!\frac{1}{2})_{n_2\!+\!1} (c\!+\!n_1\!+\!2 p_2\!+\!\frac{1}{2})_N (c\!-\!N\!-\!2 p_1\!+\!\frac{3}{2})_{n_1\!+\!n_2}    }{  (N\!-\!n_1\!-\!n_2\!-\!1)!(c\!+\!n_1\!+\!\frac{1}{2})_{N\!-\!n_1} (c\!+\!n_1\!+\!2 p_2\!+\!\frac{1}{2})_{n_1}^2  (c\!+\!2 n_1\!+\!n_2\!+\!2 p_2\!+\!\frac{3}{2})_{N\!-\!n_1\!-\!n_2\!-\!1} } \\[0.5em]
			      &\quad\times \frac{  (c\!+\!N\!+\!n_1\!+\!2 p_2\!+\!2 p_3\!+\!1)_{n_2\!+\!1} (c\!+\!2 n_1\!+\!n_2\!+\!2 p_2\!+\!2 p_3\!+\!1)_{N\!-\!n_1\!-\!n_2}    }{  (c\!+\!2 n_1\!+\!n_2\!+\!2 p_2\!+\!2 p_3\!+\!1)_{n_2\!+\!1}^2 (\!-\!2 N\!-\!2 p_1\!-\!2 p_2\!-\!2 p_3\!+\!\frac{1}{2})_{N\!-\!n_1\!-\!n_2}   }  
  \end{aligned}
\end{align*}
\begin{align*}
  \begin{aligned}
    H_{2n_1\!+\!1,2n_2\!+\!1,2N}   &= \frac{n_1!n_2! (2 p_2)_{n_1\!+\!1} (2 p_3\!+\!\frac{1}{2})_{n_2\!+\!1} (c\!+\!n_1\!+\!2 p_2\!+\!\frac{1}{2})_N (c\!-\!N\!-\!2 p_1\!+\!\frac{3}{2})_{n_1\!+\!n_2\!+\!1}  }{  (N\!-\!n_1\!-\!n_2\!-\!1)!(c\!+\!n_1\!+\!\frac{3}{2})_{N\!-\!n_1\!-\!1} (c\!+\!n_1\!+\!2 p_2\!+\!\frac{1}{2})_{n_1\!+\!1}^2  (c\!+\!2 n_1\!+\!n_2\!+\!2 p_2\!+\!\frac{5}{2})_{N\!-\!n_1\!-\!n_2\!-\!2} } \\[0.5em]
				  &\quad\times \frac{   (c\!+\!N\!+\!n_1\!+\!2 p_2\!+\!2 p_3\!+\!2)_{n_2} (c\!+\!2 n_1\!+\!n_2\!+\!2 p_2\!+\!2 p_3\!+\!2)_{N\!-\!n_1\!-\!n_2}   }{  (c\!+\!2 n_1\!+\!n_2\!+\!2 p_2\!+\!2 p_3\!+\!2)_{n_2\!+\!1}^2 (\!-\!2 N\!-\!2 p_1\!-\!2 p_2\!-\!2 p_3\!+\!\frac{1}{2})_{N\!-\!n_1\!-\!n_2\!-\!1}   }  
  \end{aligned}
\end{align*}
\begin{align*}
  \begin{aligned}
  H_{2n_1,2n_2,2N\!+\!1}     &= \frac{n_1!n_2! (2 p_2)_{n_1} (2 p_3\!+\!\frac{1}{2})_{n_2} (c\!+\!n_1\!+\!2 p_2\!+\!\frac{1}{2})_N (c\!-\!N\!-\!2 p_1\!+\!\frac{1}{2})_{n_1\!+\!n_2}      }{  (N\!-\!n_1\!-\!n_2)!  (c\!+\!n_1\!+\!\frac{1}{2})_{N\!-\!n_1\!+\!1} (c\!+\!n_1\!+\!2 p_2\!+\!\frac{1}{2})_{n_1}^2  (c\!+\!2 n_1\!+\!n_2\!+\!2 p_2\!+\!\frac{1}{2})_{N\!-\!n_1\!-\!n_2} } \\[0.5em]
				  &\quad\times \frac{  (c\!+\!N\!+\!n_1\!+\!2 p_2\!+\!2 p_3\!+\!2)_{n_2} (c\!+\!2 n_1\!+\!n_2\!+\!2 p_2\!+\!2 p_3\!+\!1)_{N\!-\!n_1\!-\!n_2\!+\!1}   }{   (c\!+\!2 n_1\!+\!n_2\!+\!2 p_2\!+\!2 p_3\!+\!1)_{n_2}^2 (\!-\!2 N\!-\!2 p_1\!-\!2 p_2\!-\!2 p_3\!-\!\frac{1}{2})_{N\!-\!n_1\!-\!n_2\!+\!1}   }  
  \end{aligned}
\end{align*}
\begin{align*}
  \begin{aligned}
  H_{2n_1\!+\!1,2n_2,2N\!+\!1}   &= \frac{n_1!n_2! (2 p_2)_{n_1\!+\!1} (2 p_3\!+\!\frac{1}{2})_{n_2} (c\!+\!n_1\!+\!2 p_2\!+\!\frac{1}{2})_N (c\!-\!N\!-\!2 p_1\!+\!\frac{1}{2})_{n_1\!+\!n_2\!+\!1}    }{  (N\!-\!n_1\!-\!n_2)!  (c\!+\!n_1\!+\!\frac{3}{2})_{N\!-\!n_1} (c\!+\!n_1\!+\!2 p_2\!+\!\frac{1}{2})_{n_1\!+\!1}^2  (c\!+\!2 n_1\!+\!n_2\!+\!2 p_2\!+\!\frac{3}{2})_{N\!-\!n_1\!-\!n_2\!-\!1} } \\[0.5em]
				  &\quad\times \frac{  (c\!+\!N\!+\!n_1\!+\!2 p_2\!+\!2 p_3\!+\!2)_{n_2} (c\!+\!2 n_1\!+\!n_2\!+\!2 p_2\!+\!2 p_3\!+\!2)_{N\!-\!n_1\!-\!n_2}   }{   (c\!+\!2 n_1\!+\!n_2\!+\!2 p_2\!+\!2 p_3\!+\!2)_{n_2}^2 (\!-\!2 N\!-\!2 p_1\!-\!2 p_2\!-\!2 p_3\!-\!\frac{1}{2})_{N\!-\!n_1\!-\!n_2}   }  
  \end{aligned}
\end{align*}
\begin{align*}
  \begin{aligned}
  H_{2n_1,2n_2\!+\!1,2N\!+\!1}   &= \frac{n_1!n_2! (2 p_2)_{n_1} (2 p_3\!+\!\frac{1}{2})_{n_2\!+\!1} (c\!+\!n_1\!+\!2 p_2\!+\!\frac{1}{2})_N (c\!-\!N\!-\!2 p_1\!+\!\frac{1}{2})_{n_1\!+\!n_2\!+\!1}   }{  (N\!-\!n_1\!-\!n_2)!  (c\!+\!n_1\!+\!\frac{1}{2})_{N\!-\!n_1\!+\!1} (c\!+\!n_1\!+\!2 p_2\!+\!\frac{1}{2})_{n_1}^2  (c\!+\!2 n_1\!+\!n_2\!+\!2 p_2\!+\!\frac{3}{2})_{N\!-\!n_1\!-\!n_2\!-\!1} } \\[0.5em]
				    &\quad\times \frac{  (c\!+\!N\!+\!n_1\!+\!2 p_2\!+\!2 p_3\!+\!1)_{n_2\!+\!1} (c\!+\!2 n_1\!+\!n_2\!+\!2 p_2\!+\!2 p_3\!+\!1)_{N\!-\!n_1\!-\!n_2}  }{  (c\!+\!2 n_1\!+\!n_2\!+\!2 p_2\!+\!2 p_3\!+\!1)_{n_2\!+\!1}^2 (\!-\!2 N\!-\!2 p_1\!-\!2 p_2\!-\!2 p_3\!-\!\frac{1}{2})_{N\!-\!n_1\!-\!n_2}  }  
  \end{aligned}
\end{align*}
\begin{align*}
  \begin{aligned}
  H_{2n_1\!+\!1,2n_2\!+\!1,2N\!+\!1} &= \frac{n_1!n_2! (2 p_2)_{n_1\!+\!1} (2 p_3\!+\!\frac{1}{2})_{n_2\!+\!1} (c\!+\!n_1\!+\!2 p_2\!+\!\frac{1}{2})_N (c\!-\!N\!-\!2 p_1\!+\!\frac{1}{2})_{n_1\!+\!n_2\!+\!1}     }{ (N\!-\!n_1\!-\!n_2\!-\!1)!(c\!+\!n_1\!+\!\frac{3}{2})_{N\!-\!n_1} (c\!+\!n_1\!+\!2 p_2\!+\!\frac{1}{2})_{n_1\!+\!1}^2  (c\!+\!2 n_1\!+\!n_2\!+\!2 p_2\!+\!\frac{5}{2})_{N\!-\!n_1\!-\!n_2\!-\!2} }  \\[0.5em]
				  &\quad\times \frac{  (c\!+\!N\!+\!n_1\!+\!2 p_2\!+\!2 p_3\!+\!2)_{n_2\!+\!1} (c\!+\!2 n_1\!+\!n_2\!+\!2 p_2\!+\!2 p_3\!+\!2)_{N\!-\!n_1\!-\!n_2}  }{   (c\!+\!2 n_1\!+\!n_2\!+\!2 p_2\!+\!2 p_3\!+\!2)_{n_2\!+\!1}^2 (\!-\!2 N\!-\!2 p_1\!-\!2 p_2\!-\!2 p_3\!-\!\frac{1}{2})_{N\!-\!n_1\!-\!n_2}   }  
  \end{aligned}
\end{align*}

The coefficients of the differential operator $\mathcal{L}$ \eqref{mathcalL} of section 3 read
\begin{align*} \label{CoefC}
\begin{aligned}
 C_{-1,-1} &= -\frac{(z_1-a) (z_1-b) (z_2-c) (z_2-d) (z_1 z_2-a_2) (z_1 z_2-a_2 q)}{\left(z_1^2-1\right) \left(z_2^2-1\right) \left(z_1^2-q\right) \left(q-z_2^2\right)} \\[0.5em]
 C_{-1, 0} &= \frac{z_2 q (q+1) \left(z_1-a\right) \left(z_1-b\right) \left(z_1z_2-a_2\right) \left(z_1-a_2 z_2\right) \left(1+\frac{c d}{q}-\frac{(z_2^2+1)(c+d)}{(q+1) z_2}\right)}{\left(z_1^2-1\right) \left(z_1^2-q\right) \left(q-z_2^2\right) \left(q z_2^2-1\right)} \\[0.5em]
 C_{-1, 1} &= \frac{(a-z_1) (z_1-b) (c z_2-1) (1-d z_2) (z_1-a_2 z_2) (z_1-a_2 q z_2)}{\left(z_1^2-1\right) \left(z_2^2-1\right) \left(z_1^2-q\right) \left(q z_2^2-1\right)} \\[0.5em]
 C_{ 0,-1} &= \frac{z_1 q (q+1) \left(z_2-c\right) \left(z_2-d\right) \left(z_1z_2-a_2\right) \left(z_2-a_2 z_1\right) \left(1+\frac{a b}{q}-\frac{(z_1^2+1)(a+b)}{z_1(q+1)}\right)}{\left(z_2^2-1\right) \left(z_1^2-q\right) \left(1-q z_1^2\right) \left(z_2^2-q\right)} \\[0.5em]
 C_{ 0, 0} &= -1+\frac{a_2 (a+b) (c+d)}{q+1} -\frac{a a_2^2 b c d}{q}  \\ &\quad             +\frac{z_1^2 z_2^2 q^2 (q+1)^2 \left(1+\frac{ab}{q}-\frac{(z_1^2+1)(a+b)}{z_1(q+1)}\right) \left(1+\frac{a_2^2}{q}-\frac{a_2 (z_1^2+1) (z_2^2+1)}{z_1z_2(q+1)}\right) \left(1+\frac{c d}{q}-\frac{(z_2^2+1)(c+d)}{z_2(q+1)}\right)   }{    \left(z_1^2-q\right) \left(1-q z_1^2\right) \left(z_2^2-q\right) \left(1-q z_2^2\right)}   \\[0.5em]
 C_{ 0, 1} &= \frac{z_1 q (q+1) (1-c z_2) (1-d z_2) (z_1-a_2 z_2) (1-a_2 z_1 z_2) \left(1+\frac{a b}{q}-\frac{(z_1^2+1)(a+b)}{z_1(q+1)}\right)}{\left(1-z_2^2\right) \left(z_1^2-q\right) \left(1-q z_1^2\right) \left(1-q z_2^2\right)} \\[0.5em]
 C_{ 1,-1} &= \frac{(a z_1-1) (b z_1-1) (c-z_2) (z_2-d) (a_2 z_1-z_2) (a_2 q z_1-z_2)}{\left(z_1^2-1\right) \left(z_2^2-1\right) \left(q z_1^2-1\right) \left(q-z_2^2\right)} \\[0.5em]
 C_{ 1, 0} &= \frac{z_2 q (q+1) (1-a z_1) (1-b z_1) (z_2-a_2 z_1) (1-a_2 z_1 z_2) \left(1+\frac{c d}{q}-\frac{(z_2^2+1) (c+d)}{(q+1) z_2}\right)}{\left(z_1^2-1\right) \left(q z_1^2-1\right) \left(q-z_2^2\right) \left(q z_2^2-1\right)} \\[0.5em]
 C_{ 1, 1} &= -\frac{(a z_1-1) (b z_1-1) (c z_2-1) (1-d z_2) (a_2 z_1 z_2-1) (a_2 q z_1 z_2-1)}{\left(z_1^2-1\right) \left(z_2^2-1\right) \left(q z_1^2-1\right) \left(q z_2^2-1\right)}
\end{aligned}
\end{align*}

The recurrence coefficients for the Askey-Wilson polynomials appearing in \eqref{RelRecAW9} have the following expressions 
\begin{align*}
  \tau^{(1)}_{n_1,n_2} &= -\frac{\left(a_2^2 q^{n_1}\!-\!1\right) \left(a a_2^2 b q^{n_1}\!-\!q\right) \left(a a_2 c q^{n_1\!+n_2}\!-\!1\right) \left(a_2 b c q^{n_1\!+n_2}\!-\!1\right) \left(a a_2^2 b c d q^{2 n_1\!+n_2}\!-\!1\right) \left(a a_2^2 b c d q^{2 n_1\!+n_2}\!-\!q\right)}{\left(a a_2^2 b q^{2 n_1}\!-\!1\right) \left(a a_2^2 b q^{2 n_1}\!-\!q\right) \left(a a_2^2 b c d q^{2 (n_1\!+n_2)}\!-\!1\right) \left(a a_2^2 b c d q^{2 (n_1\!+n_2)}\!-\!q\right)} \\[0.5em]
  \tau^{(2)}_{n_1,n_2} &= \frac{a_2 c q^{n_1} \left(q^{n_2}\!-\!1\right) \left(a_2^2 q^{n_1}\!-\!1\right) \left(a a_2^2 b q^{n_1}\!-\!q\right) \left(a a_2^2 b c d q^{2 n_1\!+n_2}\!-\!q\right) \big((a\!+\!b)(q\!+\!a a_2^2 b c d q^{2 (n_1\!+n_2)})\!-\!a a_2 b (q\!+\!1) (c\!+\!d) q^{n_1\!+n_2}\big)  }{   \left(q\!-\!a a_2^2 b q^{2 n_1}\right) \left(a a_2^2 b q^{2 n_1}\!-\!1\right) \left(q^2\!-\!a a_2^2 b c d q^{2 (n_1\!+n_2)}\right) \left(a a_2^2 b c d q^{2 (n_1\!+n_2)}\!-\!1\right)}  \\[0.5em]
  \tau^{(3)}_{n_1,n_2} &= -\frac{a a_2^2 b c^2 q^{2 n_1} \left(q^{n_2}\!-\!1\right) \left(q^{n_2}\!-\!q\right) \left(a_2^2 q^{n_1}\!-\!1\right) \left(a a_2^2 b q^{n_1}\!-\!q\right) \left(a a_2 d q^{n_1\!+n_2}\!-\!q\right) \left(a_2 b d q^{n_1\!+n_2}\!-\!q\right)}{\left(a a_2^2 b q^{2 n_1}\!-\!1\right) \left(a a_2^2 b q^{2 n_1}\!-\!q\right) \left(a a_2^2 b c d q^{2 (n_1\!+n_2)}\!-\!q\right) \left(a a_2^2 b c d q^{2 (n_1\!+n_2)}\!-\!q^2\right)} \\[0.5em]
  \tau^{(4)}_{n_1,n_2} &= -\frac{q (q\!+\!1) \left(1\!-\!c d q^{n_2}\right) \left(1\!-\!a a_2 c q^{n_1\!+n_2}\right) \left(1\!-\!a_2 b c q^{n_1\!+n_2}\right) \left(1\!-\!a a_2^2 b c d q^{2 n_1\!+n_2\!-\!1}\right) \left(1\!+\!\frac{a_2^2}{q}\!-\!\frac{(a^2b+aq) (a^2 a_2^2b q^{2 n_1}+aq)}{a^3b q^{n_1\!+1} (q+1)}\right)    }{\left(1\!-\!\frac{q^{2\!-\!2 n_1}}{a a_2^2 b}\right) \left(1\!-\!a a_2^2 b q^{2 n_1}\right) \left(1\!-\!a a_2^2 b c d q^{2 (n_1\!+n_2)}\right) \left(1\!-\!a a_2^2 b c d q^{2 n_1\!+\!2 n_2\!-\!1}\right)} \\[0.5em]
  \tau^{(5)}_{n_1,n_2} &= -\frac{q^2 (q\!+\!1)^2 \left(1\!+\!\frac{a_2^2}{q}\!-\!\frac{ (ab+q)(q+a a_2^2 b q^{2n_1} )}{a b q^{n_1\!+1} (q+1)}\right) \left(1\!+\!\frac{c}{d}\!-\!\frac{ (a+b)(q+a a_2^2 bcdq^{2(\!n_1\!+n_2\!)}) }{aa_2bd q^{n_1\!+n_2} (q+1)}\right)     \left(1\!+\!\frac{c d}{q}\!-\!\frac{ (q+aa_2^2bq^{2n_1})(q+aa_2^2bcdq^{2(n_1\!+n_2)})   }{ a a_2^2 b q^{1+\!2n_1\!+n_2}(q+1)}\right)      }{        \left(1\!-\!\frac{q^{2\!-\!2 n_1}}{a a_2^2 b}\right) \left(1\!-\!a a_2^2 b q^{2 n_1}\right) \left(1\!-\!\frac{q^{\!-\!2 (n_1\!+n_2\!-\!1)}}{a a_2^2 b c d}\right) \left(1\!-\!a a_2^2 b c d q^{2 (n_1\!+n_2)}\right)} \\[0.5em]
  \tau^{(6)}_{n_1,n_2} &= \frac{a_2^2 c^2 q^{n_1}\! \left(q^{n_2}\!-\!1\right) (q\!-\!a a_2^2 b q^{2 n_1\!+n_2}) \left(q\!-\!a a_2 d q^{n_1\!+n_2}\right) \left(q\!-\!a_2 b d q^{n_1\!+n_2}\right) (abq^{n_1}(1\!+\!q)(a_2^2\!+\!q)\!-\!(ab\!+\!q)(q\!+\!aa_2^2bq^{2n_1}))    }{   \left(q^2\!-\!a a_2^2 b q^{2 n_1}\right) \left(a a_2^2 b q^{2 n_1}\!-\!1\right) \left(q\!-\!a a_2^2 b c d q^{2 (n_1\!+n_2)}\right) \left(q^2\!-\!a a_2^2 b c d q^{2 (n_1\!+n_2)}\right)}  \\[0.5em]
  \tau^{(7)}_{n_1,n_2} &= \frac{a a_2^4 b q^{2 n_1\!+\!1} \left(1\!-\!q^{n_1}\right) \left(q\!-\!a b q^{n_1}\right) \left(1\!-\!c d q^{n_2}\right) \left(1\!-\!c d q^{n_2\!+\!1}\right) \left(1\!-\!a a_2 c q^{n_1\!+n_2}\right) \left(1\!-\!a_2 b c q^{n_1\!+n_2}\right)}{\left(q\!-\!a a_2^2 b q^{2 n_1}\right) \left(q^2\!-\!a a_2^2 b q^{2 n_1}\right) \left(q\!-\!a a_2^2 b c d q^{2 (n_1\!+n_2)}\right) \left(a a_2^2 b c d q^{2 (n_1\!+n_2)}\!-\!1\right)} \\[0.5em]
  \tau^{(8)}_{n_1,n_2} &= \frac{a_2^3 c q^{n_1} \left(q^{n_1}\!-\!1\right) \left(q\!-\!a b q^{n_1}\right) \left(c d q^{n_2}\!-\!1\right) \left(q\!-\!a a_2^2 b q^{2 n_1\!+n_2}\right) \left((a+b)(q+aa_2^2bcdq^{2 (n_1\!+n_2)})-aa_2b(c\!+\!d)(1\!+\!q)q^{n_1\!+n_2}\right)      }{    \left(q\!-\!a a_2^2 b q^{2 n_1}\right) \left(q^2\!-\!a a_2^2 b q^{2 n_1}\right) \left(q^2\!-\!a a_2^2 b c d q^{2 (n_1\!+n_2)}\right) \left(1\!-\!a a_2^2 b c d q^{2 (n_1\!+n_2)}\right)} \\[0.5em]
  \tau^{(9)}_{n_1,n_2} &= -\frac{a_2^2 c^2 \left(q^{n_1}\!-\!1\right) \left(a b q^{n_1}\!-\!q\right) \left(a a_2^2 b q^{2 n_1\!+\!n_2}\!-\!q\right) \left(a a_2^2 b q^{2 n_1\!+\!n_2}\!-\!q^2\right) \left(a a_2 d q^{n_1\!+\!n_2}\!-\!q\right) \left(a_2 b d q^{n_1\!+\!n_2}\!-\!q\right)}{\left(q\!-\!a a_2^2 b q^{2 n_1}\right) \left(q^2\!-\!a a_2^2 b q^{2 n_1}\right) \left(q\!-\!a a_2^2 b c d q^{2 (n_1\!+\!n_2)}\right) \left(q^2\!-\!a a_2^2 b c d q^{2 (n_1\!+\!n_2)}\right)}
\end{align*}

The coefficients for the 9-term recurrence relation \eqref{RelRecBI9} satisfied by the bivariate Bannai-Ito polynomials are
\begin{align*}
 \theta^{(1)}_{n_1,n_2} = (-1)^{n_2} 
\end{align*}
\begin{align*}
 \theta^{(2)}_{n_1,n_2} = \begin{cases}
                           \frac{n_2}{4} \left(\frac{2 \beta +2 \gamma +2 \epsilon+n_1+n_2+1}{-\alpha +\beta +\gamma +\delta +2 \epsilon+n_1+n_2+1}-\frac{-2 \alpha +2 \gamma +2 \epsilon+n_1+n_2}{-\alpha +\beta +\gamma +\delta +2 \epsilon+n_1+n_2}\right) \quad&\text{$n_1$ even, $n_2$ even,} \\[0.5em]     
                           \frac{1}{4} (2 \gamma +2 \delta +n_2) \left(\frac{-2 \alpha +2 \gamma +2 \epsilon+n_1+n_2+1}{-\alpha +\beta +\gamma +\delta +2 \epsilon+n_1+n_2+1}-\frac{2 \beta +2 \gamma +2 \epsilon+n_1+n_2}{-\alpha +\beta +\gamma +\delta +2 \epsilon+n_1+n_2}\right) \quad&\text{$n_1$ even, $n_2$ odd,} \\[0.5em]  
			   \frac{ n_2}{4} \left(\frac{2 \beta +2 \gamma +2 \epsilon+n_1+n_2}{-\alpha +\beta +\gamma +\delta +2 \epsilon+n_1+n_2}-\frac{-2 \alpha +2 \gamma +2 \epsilon+n_1+n_2+1}{-\alpha +\beta +\gamma +\delta +2 \epsilon+n_1+n_2+1}\right) \quad&\text{$n_1$ odd, $n_2$ even,} \\[0.5em]  
			   \frac{1}{4} (2 \gamma +2 \delta +n_2) \left(\frac{-2 \alpha +2 \gamma +2 \epsilon+n_1+n_2}{-\alpha +\beta +\gamma +\delta +2 \epsilon+n_1+n_2}-\frac{2 \beta +2 \gamma +2 \epsilon+n_1+n_2+1}{-\alpha +\beta +\gamma +\delta +2 \epsilon+n_1+n_2+1}\right) \quad&\text{$n_1$ odd, $n_2$ odd,}  
\end{cases}
\end{align*}
\begin{align*}
 \theta^{(3)}_{n_1,n_2} = \begin{cases}
                           -\frac{n_2 (2 \gamma +2 \delta +n_2-1) (-2 \alpha +2 \gamma +2 \epsilon+n_1+n_2) (2 \beta +2 \delta +2 \epsilon+n_1+n_2)}{16 (-\alpha +\beta +\gamma +\delta +2 \epsilon+n_1+n_2)^2} \quad&\text{$n_1$ even, $n_2$ even,} \\[0.5em]     
                \frac{(n_2-1) (2 \gamma +2 \delta +n_2) (-2 \alpha +2 \delta +2 \epsilon+n_1+n_2) (2 \beta +2 \gamma +2 \epsilon+n_1+n_2)}{16 (-\alpha +\beta +\gamma +\delta +2 \epsilon+n_1+n_2)^2} \quad&\text{$n_1$ even, $n_2$ odd,} \\[0.5em]  
			   -\frac{n_2 (2 \gamma +2 \delta +n_2-1) (-2 \alpha +2 \delta +2 \epsilon+n_1+n_2) (2 \beta +2 \gamma +2 \epsilon+n_1+n_2)}{16 (-\alpha +\beta +\gamma +\delta +2 \epsilon+n_1+n_2)^2} \quad&\text{$n_1$ odd, $n_2$ even,} \\[0.5em]  
			   \frac{(n_2-1) (2 \gamma +2 \delta +n_2) (-2 \alpha +2 \gamma +2 \epsilon+n_1+n_2) (2 \beta +2 \delta +2 \epsilon+n_1+n_2)}{16 (-\alpha +\beta +\gamma +\delta +2 \epsilon+n_1+n_2)^2} \quad&\text{$n_1$ odd, $n_2$ odd,}  
\end{cases}
\end{align*}
\begin{align*}
 \theta^{(4)}_{n_1,n_2} = \begin{cases}
                          1-\frac{2 \epsilon+\frac{n_1}{2}}{-\alpha +\beta +2 \epsilon+n_1+\frac{1}{2}}-\frac{\frac{n_1}{2}}{ -\alpha +\beta +2 \epsilon+n_1-\frac{1}{2}} \quad&\text{$n_1$ even,} \\[0.5em]  
			  1-\frac{2 \epsilon+\frac{n_1-1}{2}}{-\alpha +\beta +2 \epsilon+n_1-\frac{1}{2}}-\frac{\frac{n_1+1}{2}}{ -\alpha +\beta +2 \epsilon+n_1+\frac{1}{2}} \quad&\text{$n_1$ odd,}  
\end{cases}
\end{align*}
\begin{align*}
 \theta^{(5)}_{n_1,n_2} = \begin{cases}
                          &-\frac{1}{4}\left(\frac{n_1}{-2 \alpha +2 \beta +4 \epsilon+2 n_1-1}  +\frac{4 \epsilon+n_1}{-2 \alpha +2 \beta +4 \epsilon+2 n_1+1}-1\right) \left(2 \alpha   \!+\!2 \beta \!-\frac{2 \beta +2 \gamma +2 \epsilon+n_1+n_2+1}{-\alpha +\beta +\gamma +\delta +2 \epsilon+n_1+n_2+1}+\!1\right)     \\[0.5em]  &\hspace{0cm} \times\phantom{\frac14} \left(-2 \alpha \!+\!2 \beta \!-\frac{n_2}{-\alpha +\beta +\gamma +\delta +2 \epsilon+n_1+n_2}+\!4\epsilon\!+\!2 n_1\!+\!1\right)     \hfill \text{$n_1$ even, $n_2$ even,} \\[0.8em]     
                          &-\frac{1}{4}\left(\frac{n_1}{-2 \alpha +2 \beta +4 \epsilon+2 n_1-1}  +\frac{4 \epsilon+n_1}{-2 \alpha +2 \beta +4 \epsilon+2 n_1+1}-1\right) \left(2 \alpha   \!+\!2 \beta \!+\frac{2 (\beta +\gamma )+2 \epsilon+n_1+n_2}{-\alpha +\beta +\gamma +\delta +2 \epsilon+n_1+n_2}-\!1\right)         \\[0.5em]  &\hspace{0cm} \times\phantom{\frac14} \left(-2 \alpha \!+\!2 \beta \!+\frac{n_2+1}{-\alpha +\beta +\gamma +\delta +2 \epsilon+n_1+n_2+1}+\!4\epsilon\!+\!2 n_1\!-\!1\right) \hfill \text{$n_1$ even, $n_2$ odd,} \\[0.8em]  
			  &-\frac{1}{4}\left(\frac{n_1+1}{-2 \alpha +2 \beta +4 \epsilon+2 n_1+1}+\frac{4 \epsilon+n_1-1}{-2 \alpha +2 \beta +4 \epsilon+2 n_1-1}-1\right) \left(2 \alpha \!+\!2 \beta \!+\frac{-2 \alpha +2 \gamma +2 \epsilon+n_1+n_2+1}{-\alpha +\beta +\gamma +\delta +2 \epsilon+n_1+n_2+1}-\!1\right) \\[0.5em]  &\hspace{0cm} \times\phantom{\frac14} \left(-2 \alpha \!+\!2 \beta \!-\frac{n_2}{-\alpha +\beta +\gamma +\delta +2 \epsilon+n_1+n_2}+\!4\epsilon\!+\!2 n_1\!+\!1\right)       \hfill \text{$n_1$ odd, $n_2$ even,} \\[0.8em]  
			  &-\frac{1}{4}\left(\frac{n_1+1}{-2 \alpha +2 \beta +4 \epsilon+2 n_1+1}+\frac{4 \epsilon+n_1-1}{-2 \alpha +2 \beta +4 \epsilon+2 n_1-1}-1\right) \left(2 \alpha \!+\!2 \beta \!-\frac{-2 \alpha +2 \gamma +2 \epsilon+n_1+n_2}{-\alpha +\beta +\gamma +\delta +2 \epsilon+n_1+n_2}+\!1\right)     \\[0.5em]  &\hspace{0cm} \times\phantom{\frac14} \left(-2 \alpha \!+\!2 \beta \!+\frac{n_2+1}{-\alpha +\beta +\gamma +\delta +2 \epsilon+n_1+n_2+1}+\!4\epsilon\!+\!2 n_1\!-\!1\right)   \hfill \text{$n_1$ odd, $n_2$ odd,}  
\end{cases}
\end{align*}
\begin{align*}
 \theta^{(6)}_{n_1,n_2} = \begin{cases}
                          &\frac{n_2 \left(2 \alpha -2 \beta +\frac{2 n_1}{-2 \alpha +2 \beta +4 \epsilon+2 n_1-1}-1\right) \left(-\alpha +\gamma +\epsilon+\frac{n_1+n_2}{2}\right) \left(\beta +\delta +\epsilon+\frac{n_1+n_2}{2}\right) \left(-\alpha +\beta +\gamma +\delta +2 \epsilon+n_1+\frac{n_2}{2}\right)}{2 (-2 \alpha +2 \beta +4 \epsilon+2 n_1+1) (-\alpha +\beta +\gamma +\delta +2 \epsilon+n_1+n_2)^2} \\ &\hfill \text{$n_1$ even, $n_2$ even,} \\[0.5em]     
                          &\frac{ \left(\gamma +\delta +\frac{n_2}{2}\right)\left(2 \alpha -2 \beta +\frac{2 n_1}{-2 \alpha +2 \beta +4 \epsilon+2 n_1-1}-1\right) \left(-\alpha +\beta +2 \epsilon+n_1+\frac{n_2}{2}\right) \left(-\alpha +\delta +\epsilon+\frac{n_1+n_2}{2}\right) \left(\beta +\gamma +\epsilon+\frac{n_1+n_2}{2}\right)}{(-2 \alpha +2 \beta +4 \epsilon+2 n_1+1) (-\alpha +\beta +\gamma +\delta +2 \epsilon+n_1+n_2)^2} \\ &\hfill \text{$n_1$ even, $n_2$ odd,} \\[0.5em]  
			  &\frac{n_2 \left(2 \alpha -2 \beta +\frac{8 \epsilon+2 n_1-2}{-2 \alpha +2 \beta +4 \epsilon+2 n_1-1}-1\right) \left(-\alpha +\delta +\epsilon+\frac{n_1+n_2}{2}\right) \left(\beta +\gamma +\epsilon+\frac{n_1+n_2}{2}\right) \left(-\alpha +\beta +\gamma +\delta +2 \epsilon+n_1+\frac{n_2}{2}\right)}{2 (-2 \alpha +2 \beta +4 \epsilon+2 n_1+1) (-\alpha +\beta +\gamma +\delta +2 \epsilon+n_1+n_2)^2} \\ &\hfill \text{$n_1$ odd, $n_2$ even,} \\[0.5em]  
			  &\frac{ \left(\gamma +\delta +\frac{n_2}{2}\right)\left(2 \alpha -2 \beta +\frac{8 \epsilon+2 n_1-2}{-2 \alpha +2 \beta +4 \epsilon+2 n_1-1}-1\right)  \left(-\alpha +\beta +2 \epsilon+n_1+\frac{n_2}{2}\right) \left(-\alpha +\gamma +\epsilon+\frac{n_1+n_2}{2}\right) \left(\beta +\delta +\epsilon+\frac{n_1+n_2}{2}\right)}{(-2 \alpha +2 \beta +4 \epsilon+2 n_1+1) (-\alpha +\beta +\gamma +\delta +2 \epsilon+n_1+n_2)^2} \\ &\hfill \text{$n_1$ odd, $n_2$ odd,}  
\end{cases}
\end{align*}
\begin{align*}
 \theta^{(7)}_{n_1,n_2} = \begin{cases}
                         (-1)^{n_2}\frac{ n_1 (-2 \alpha +2 \beta +4 \epsilon+n_1-1)}{(-2 \alpha +2 \beta +4 \epsilon+2 n_1-1)^2} \quad&\text{$n_1$ even,} \\[0.5em]  
			  (-1)^{n_2}\frac{(4 \epsilon+n_1-1) (-2 \alpha +2 \beta +n_1)}{(-2 \alpha +2 \beta +4 \epsilon+2 n_1-1)^2} \quad&\text{$n_1$ odd,}  
\end{cases}
\end{align*}
\begin{align*}
 \theta^{(8)}_{n_1,n_2} = \begin{cases}
                          -\frac{ \frac{n_1}{2} (-2 \alpha +2 \beta +4 \epsilon+n_1-1) \left(-\alpha +\beta +\gamma +\delta +2 \epsilon+n_1+\frac{n_2}{2}\right) \left(2 \alpha +2 \beta -\frac{-2\alpha +2\gamma +2\epsilon+n_1+n_2}{-\alpha +\beta +\gamma +\delta +2 \epsilon+n_1+n_2}+1\right)    }{ (-2 \alpha +2 \beta +4 \epsilon+2 n_1-1)^2 (-\alpha +\beta +\gamma +\delta +2 \epsilon+n_1+n_2+1)} \quad&\text{$n_1$ even, $n_2$ even,} \\[0.5em]     
                          \frac{ \frac{n_1}{2} (-2 \alpha +2 \beta +4 \epsilon+n_1-1) \left(-\alpha +\beta +2 \epsilon+n_1+\frac{n_2}{2}\right) \left(2 \alpha +2 \beta +\frac{ -2\alpha +2\gamma +2\epsilon+n_1+n_2+1}{-\alpha +\beta +\gamma +\delta +2 \epsilon+n_1+n_2+1}-1\right)             }{ (-2 \alpha +2 \beta +4 \epsilon+2 n_1-1)^2 (-\alpha +\beta +\gamma +\delta +2 \epsilon+n_1+n_2)} \quad&\text{$n_1$ even, $n_2$ odd,} \\[0.5em]  
			  \frac{\left(2 \epsilon+\frac{n_1-1}{2}\right) (2 \alpha -2 \beta -n_1) \left(-\alpha +\beta +\gamma +\delta +2 \epsilon+n_1+\frac{n_2}{2}\right) \left(2 \alpha +2 \beta +\frac{-2\alpha +2\gamma +2\epsilon+n_1+n_2+1}{-\alpha +\beta +\gamma +\delta +2 \epsilon+n_1+n_2+1}-1\right)    }{(-2 \alpha +2 \beta +4 \epsilon+2 n_1-1)^2 (-\alpha +\beta +\gamma +\delta +2 \epsilon+n_1+n_2)} \quad&\text{$n_1$ odd, $n_2$ even,} \\[0.5em]  
			   -\frac{\left(2 \epsilon+\frac{n_1-1}{2}\right) (2 \alpha -2 \beta -n_1) \left(-\alpha +\beta +2 \epsilon+n_1+\frac{n_2}{2}\right) \left(2 \alpha +2 \beta -\frac{-2\alpha +2\gamma +2\epsilon+n_1+n_2}{-\alpha +\beta +\gamma +\delta +2 \epsilon+n_1+n_2}+1\right)        }{(-2 \alpha +2 \beta +4 \epsilon+2 n_1-1)^2 (-\alpha +\beta +\gamma +\delta +2 \epsilon+n_1+n_2+1)} \quad&\text{$n_1$ odd, $n_2$ odd,}  
\end{cases}
\end{align*}
\begin{align*}
 \theta^{(9)}_{n_1,n_2} = 
\begin{cases}
                      &\!\!\!\!\!\frac{ \frac{n_1}{2}(-2 \alpha +2 \beta +4 \epsilon+n_1-1) (-2 \alpha +2 \beta +4 \epsilon+2 n_1+n_2-1) \left(-\alpha +\gamma +\epsilon+\frac{n_1+n_2}{2}\right) \left(\beta +\delta +\epsilon+\frac{n_1+n_2}{2}\right) \left(-\alpha +\beta +\gamma +\delta +2 \epsilon+n_1+\frac{n_2}{2}\right)   }{ (-1)^{n_2+1} (-2 \alpha +2 \beta +4 \epsilon+2 n_1-1)^2 (-\alpha +\beta +\gamma +\delta +2 \epsilon+n_1+n_2)^2        } \\ &\hfill \text{$n_1$ even, $n_2$ even,} \\[0.5em]     
                      &\!\!\!\!\!\frac{ \frac{n_1}{2}(-2 \alpha +2 \beta +4 \epsilon+n_1-1) (-2 \alpha +2 \beta +4 \epsilon+2 n_1+n_2) \left(-\alpha +\delta +\epsilon+\frac{n_1+n_2}{2}\right) \left(\beta +\gamma +\epsilon+\frac{n_1+n_2}{2}\right) \left(-\alpha +\beta +\gamma +\delta +2 \epsilon+n_1+\frac{n_2-1}{2}\right)   }{ (-1)^{n_2+1} (-2 \alpha +2 \beta +4 \epsilon+2 n_1-1)^2 (-\alpha +\beta +\gamma +\delta +2 \epsilon+n_1+n_2)^2        } \\ &\hfill \text{$n_1$ even, $n_2$ odd,} \\[0.5em]  
                      &\!\!\!\!\!\frac{\left(2 \epsilon+\frac{n_1-1}{2}\right) (-2 \alpha +2 \beta +n_1) (-2 \alpha +2 \beta +4 \epsilon+2 n_1+n_2-1) \left(-\alpha +\delta +\epsilon+\frac{n_1+n_2}{2}\right) \left(\beta +\gamma +\epsilon+\frac{n_1+n_2}{2}\right) \left(-\alpha +\beta +\gamma +\delta +2 \epsilon+n_1+\frac{n_2}{2}\right)}{(-1)^{n_2+1}(-2 \alpha +2 \beta +4 \epsilon+2 n_1-1)^2 (-\alpha +\beta +\gamma +\delta +2 \epsilon+n_1+n_2)^2} \\ &\hfill \text{$n_1$ odd, $n_2$ even,} \\[0.5em]  
		      &\!\!\!\!\!\frac{\left(2 \epsilon+\frac{n_1-1}{2}\right) (-2 \alpha +2 \beta +n_1) (-2 \alpha +2 \beta +4 \epsilon+2 n_1+n_2) \left(-\alpha +\gamma +\epsilon+\frac{n_1+n_2}{2}\right) \left(\beta +\delta +\epsilon+\frac{n_1+n_2}{2}\right) \left(-\alpha +\beta +\gamma +\delta +2 \epsilon+n_1+\frac{n_2-1}{2}\right)}{(-1)^{n_2+1}(-2 \alpha +2 \beta +4 \epsilon+2 n_1-1)^2 (-\alpha +\beta +\gamma +\delta +2 \epsilon+n_1+n_2)^2} \\ &\hfill \text{$n_1$ odd, $n_2$ odd.}  
\end{cases}
\end{align*}

\bibliographystyle{elsarticle-num}
\bibliography{bivariate_BI_JMP.bib}

\end{document}